\documentclass[preprint]{elsarticle}
\usepackage[english]{babel}
\usepackage{amsfonts}
\usepackage{amsthm}
\usepackage{amsmath}
\usepackage{amscd}
\usepackage[active]{srcltx} 
\usepackage{latexsym}
\usepackage{amssymb}
\usepackage[latin1]{inputenc}
\usepackage{xcolor}
\usepackage{enumitem}

\usepackage{stackengine}
\stackMath

\numberwithin{equation}{section}

\newtheorem{thm}{Theorem}[section]
\newtheorem{prop}[thm]{Proposition}
\newtheorem{cor}[thm]{Corollary}
\newtheorem{lema}[thm]{Lemma}
\newtheorem{hyp}[thm]{Hypotheses}
\newtheorem{hyp2}[thm]{Hypothesis}

\newtheorem{prob}{Problem}
\newtheorem{innercustomthm}{Problem}
\newenvironment{customprob}[1]
  {\renewcommand\theinnercustomthm{#1}\innercustomthm}
  {\endinnercustomthm}

\newtheorem*{Def}{Definition}
\newtheorem{example}{Example}
\newtheorem{obs}[thm]{Remark}

\newcommand{\PI}[2]{\left\langle \,#1 , #2\, \right\rangle}

\newcommand{\K}[2]{\left[ \,#1 , #2\, \right]}
\newcommand{\modulo}[1]{\left| \,#1\, \right|}
\newcommand{\set}[1]{\left\{ \,#1\, \right\}}
\newcommand{\setb}[1]{\big\{ \,#1\, \big\}}
\newcommand{\setB}[1]{\Big\{ \,#1\, \Big\}}

\newcommand{\parentesis}[1]{\left( \,#1\, \right)}

\newcommand{\ra}{\rightarrow}

\newcommand{\x}{\times}

\newcommand{\CC}{\mathbb{C}}
\newcommand{\RR}{\mathbb{R}}
\newcommand{\NN}{\mathbb{N}}
\newcommand{\St}{\mathcal{S}}

\newcommand{\N}{\mathcal{N}}

\newcommand{\M}{\mathcal{M}}

\newcommand{\HH}{\mathcal{H}}
\newcommand{\KK}{\mathcal{K}}
\newcommand{\EE}{\mathcal{E}}
\newcommand{\ZZ}{\mathcal{Z}}
\newcommand{\mc}[1]{\mathcal{#1}}

\newcommand{\ort}{{[\bot]}}

\newcommand{\noi}{\noindent}

\newcommand{\CV}{\mathcal{C}_V}

\newcommand{\PV}{\mathcal{P}^+(V)}
\newcommand{\PVV}{\mathcal{P}^-(V)}
\newcommand{\CT}{\mathcal{C}_{T}}

\newcommand{\dx}{\Delta x}
\newcommand{\wt}[1]{\widetilde{#1}}
\newcommand{\la}{\lambda}

\makeatletter
\newcommand{\eqnum}{\refstepcounter{equation}\textup{\tagform@{\theequation}}}
\makeatother

\DeclareMathOperator{\real}{Re}

\begin{document}

\begin{frontmatter}

\title{Indefinite least squares with a quadratic constraint}

\author[IAM,UBA]{Santiago Gonzalez Zerbo} 
\ead{sgzerbo@fi.uba.ar}

\author[IAM,UBA]{Alejandra Maestripieri}
\ead{amaestri@fi.uba.ar}

\author[IAM,UNLP]{Francisco Mart\'{\i}nez Per\'{\i}a}
\ead{francisco@mate.unlp.edu.ar}

\address[IAM]{Instituto Argentino de Matem\'atica ``Alberto P. Calder\'on''\\ Saavedra 15, Piso 3 (1083) Buenos Aires, Argentina}
\address[UBA]{Departamento de Matem\'atica-- Facultad de Ingenier\'{\i}a -- Universidad de Buenos Aires\\ Paseo Col\'on 850 (1063) Buenos Aires, Argentina}
\address[UNLP]{Centro de Matem\'atica de La Plata -- Facultad de Ciencias Exactas -- Universidad Nacional de La Plata\\ CC 172 (1900) La Plata, Argentina}

\begin{abstract}
An abstract indefinite least squares problem with a quadratic constraint is considered. This is a
quadratic programming problem with one quadratic equality constraint, where neither the objective nor the constraint are convex functions. 
Necessary and sufficient conditions are found for the existence of solutions.
\end{abstract}


\end{frontmatter}

\section{Introduction}

Quadratic optimization is a fundamental problem in optimization theory and its applications. Economic equilibrium, combinatorial optimization and numerical
partial differential equations are all sources of quadratic optimization problems.
Quadratic programming (QP) with a
convex objective function was shown to be polynomial-time solvable. However, QP with an indefinite quadratic term is NP-hard in general.
Usually, duality concepts and variational methods are applied to characterize and compute global minimizers.
The literature on quadratically constrained quadratic programming (QCQP) problems is abundant, specially in the finite dimensional setting
\cite{Pencils,Park,Polik,Powell,Ye}. In this case, these problems
can be written in the following form:
\begin{alignat*}{3}
& \text{minimize}   \quad && f_0(x)&&=x^TP_0x+q_0^Tx+r_0\\
& \text{subject to}   \quad && f_i(x)&&=x^TP_ix+q_i^Tx+r_i\leq 0,\ \ i=1,2,...,m
\end{alignat*}
where $x\in\RR^n$ is the optimization variable, and $P_i\in\RR^{n\x n}$, $q_i\in\RR^n$, $r_i\in\RR$ are given problem
data, for $i=0,1,...,m$. 

This kind of QCQP problems can also be posed in the infinite dimensional setting, in particular in reproducing kernel Hilbert spaces (RKHS), see e.g. \cite{Dong, Kelly,Signoretto}. There, these problems stand as
\begin{alignat*}{3}
& \text{minimize}   \quad && f(x)&&=\langle T_0x,x\rangle+\langle c,x\rangle + \alpha_0\\
& \text{subject to}   \quad && g_i(x)&&=\langle T_i x,x\rangle+\langle y_i,x\rangle+\alpha_i\leq 0,\ \ i=1,2,...,m,
\end{alignat*}
where the optimization variable $x$ varies in a complex Hilbert space $(\HH,\PI{\cdot}{\cdot})$, and the data is composed of bounded operators $T_i:\HH\ra\HH$, vectors $y_i\in\HH$ and scalars $\alpha_i\in\CC$, for $i=0,1,...,m$.

On the one hand, if the operators $T_i$ are positive semidefinite, then the objetive and the restriction are
convex functions and the problem can be solved using a generalized Lagrangian and a dual maximization problem,
with the Karush-Kuhn-Tucker conditions, see e.g. \cite{Boltyanski,Boyd,Rockafellar,Sundaram}. 

On the other hand, if the operators $T_i$ are neither positive nor negative semidefinite, then the objetive and the
restrictions are not convex. Since the definiteness of the inner product in $\HH$ plays no role at all, the aim of
this work is to pose a similar QCQP problem with only one quadratic equality constraint (QP1QEC), but using indefinite inner product
spaces as codomains of the operators involved. More precisely, this paper is devoted to studying the following abstract indefinite least squares problem (ILSP) with a quadratic constraint:

\begin{prob}\label{pb 1}
Given a Hilbert space $(\HH,\PI{\cdot}{\cdot})$, and Krein spaces $(\KK,\K{\cdot}{\cdot}_\KK)$ and $(\EE,\K{\cdot}{\cdot}_{\EE})$,
let $T:\HH\ra \KK$ and $V:\HH \ra\EE$ be bounded operators. Also, assume that $T$ has closed range and $V$ is surjective. 
Given $(w_0,z_0)\in\KK\x\EE$, analyze the existence of
\begin{equation*}
\min\,\K{Tx-w_0}{Tx-w_0}_\KK,\textit{ \normalfont{subject to} }\K{Vx-z_0}{Vx-z_0}_{\EE}=0,
\end{equation*}
and if the minimum exists, find the set of arguments at which it is attained.
\end{prob}

One motivation for studying this problem is related with practical issues derived from machine learning theory. The classical literature is formulated in RKHS, and the positive definiteness of the kernel implies that the objective functions involved in the QCQP are convex, see \cite{Gartner,Mohri,Xu}. However, the main obstacle arising in the applications is to achieve the Mercer condition for the kernel, i.e. to verify that the kernel is positive definite. Numerically, this is a painful condition to verify. In \cite{Canu2004,Canusplines,Oglic_2, Oglic} different authors propose to use reproducing kernel Krein spaces (RKKS) instead of RKHS (avoiding the necessity of verifying the Mercer condition), which turns into a more efficient solving tool from the numerical point of view.
The indefinite kernel techniques have been also applied to pattern recognition problems, see \cite{Haasdonk2005,Sonnenberg2006}.

Since $\K{\cdot}{\cdot}_\KK$ and $\K{\cdot}{\cdot}_\EE$ are indefinite inner products, the objective function $x\mapsto \K{Tx-w_0}{Tx-w_0}_\KK$ is not
convex while the equality constraint $\K{Vx-z_0}{Vx-z_0}_\EE=0$ is sign indefinite.  

If $(\EE,  \K{\cdot}{\cdot}_\EE)$ is a Hilbert space, the above constrained ILSP consists in analyzing the existence of
\begin{equation*}
\min\,\K{Tx-w_0}{Tx-w_0}_\KK,\textit{ \normalfont{subject to} }Vx=z_0.
\end{equation*}
In this case, the quadratic form $x\mapsto \K{Tx-w_0}{Tx-w_0}_\KK$ is minimized over the affine manifold $x_0 + N(V)$ where $x_0\in\HH$ is a solution to $Vx=z_0$, see \cite{GMMP10_2, GMMP16}. 

In the general setting, the objective function is minimized over a set given by a quadratic constraint.
Denote $\CV$ the set of neutral elements of the quadratic form $x\mapsto \K{Vx}{Vx}_\EE$, i.e.
\begin{equation*}
	\CV=\set{u\in\HH: \K{Vu}{Vu}_\EE=0}.
\end{equation*}
Then, given any $x_0\in\HH$ such that $Vx_0=z_0$, Problem \ref{pb 1} can be restated in the following way: 
\begin{customprob}{1'}
Given $x_0\in\HH$ and $w_0\in\KK$, analyze the existence of
\begin{equation*}
\min_{y\in\CV} \K{T(x_0 + y)-w_0}{T(x_0 + y)-w_0}_\KK,
\end{equation*}
and if the minimum exists, find the set of arguments at which it is attained.
\end{customprob}
A significant difficulty that arises is that $\CV$ is not a convex set. Moreover, the convex hull of $\CV$ is the complete Hilbert space $\HH$, thus replacing $\CV$ by its convex hull trivializes the problem.

\medskip

The paper is organized as follows. Section \ref{preliminaries} introduces the notation used along the work, as well as a brief exposition on Krein spaces and linear operators on Krein spaces. Its main purpose is to present in Proposition \ref{prop:asimov} a version of Farkas' Lemma (or $S$-procedure), and some of its consequences that are used repeatedly. Given linear operators $T:\HH\ra\KK$ and $V:\HH\ra\EE$ acting between Krein spaces, let $T^\#$ and $V^\#$ denote the adjoints of $T$ and $V$, respectively, with respect to the indefinite inner products. If the quadratic form $x\mapsto \K{Vx}{Vx}_\EE$ is indefinite, Proposition \ref{prop:asimov} says that $T$ maps $\CV$ into a nonnegative set of $\KK$ if and only if there exists $\rho\in\RR$ such that $T^\#T + \rho V^\#V$ is positive semidefinite. Moreover, if such $\rho$ exists, there is a closed interval $[\rho_-,\rho_+]$ of admissible values for $\rho$. If $\mc{P}^\pm(V)$ denote the subsets of $\HH$ where the quadratic form $x\mapsto \K{Vx}{Vx}_\EE$ takes positive and negative values, respectively, the extremal values $\rho_\pm$ are determined by
\[
\rho_-:=-\inf_{x\in\mc{P}^+(V)}\frac{\K{Tx}{Tx}}{\K{Vx}{Vx}} \quad\quad\text{and}\quad\quad \rho_+:=-\sup_{x\in\mc{P}^-(V)}\frac{\K{Tx}{Tx}}{\K{Vx}{Vx}},
\]
 see Corollary \ref{cor:intervalo}.

Section \ref{least squares} starts describing under which conditions the objective function is bounded from below over the set $x_0+\CV$, see Proposition \ref{prop infimo}. This implies that in order to have solutions to Problem \ref{pb 1} it is necessary that $T(\CV)$ is a nonnegative set of $\KK$. The rest of the section is devoted to presenting necessary and sufficient conditions for the existence of solutions to Problem \ref{pb 1} for a fixed initial data $(w_0,z_0)\in\KK\x\EE$, see Proposition \ref{prop cond 1} and Theorem \ref{teo:derivada}. 

Along Section \ref{condiciones sufi} we find a set of necessary and sufficient conditions for the existence
of solutions to Problem \ref{pb 1} for every initial data $(w_0,z_0)\in\KK\x\EE$. We start by showing that $T$ mapping $\CV$ into a uniformly positive
subset of $\KK$ is a necessary condition. Although it is not enough for our purposes, it leads us into an extra necessary condition: the attainment of 
\begin{equation*}
\sup_{x\in\mc{P}^-(V)}\frac{\K{Tx}{Tx}}{\K{Vx}{Vx}}\quad\quad\text{and}\quad\quad\inf_{x\in\mc{P}^+(V)}\frac{\K{Tx}{Tx}}{\K{Vx}{Vx}}.
\end{equation*}

Finally, we show that the above condition together with $T(\CV)$ being a uniformly positive set of $\KK$ are not only necessary but sufficient for
the existence of solutions for every initial data $(w_0,z_0)\in\KK\x\EE$. This result is stated in Theorem \ref{teo:ida_y_vuelta}.

In Section \ref{soluciones} we present a full description of $\mc{Z}(w_0,z_0)$. By Theorem \ref{teo:derivada}, given $(w_0,z_0)\in\KK\x\EE$ 
the set of solutions to Problem \ref{pb 1} is
\[
\ZZ(w_0,z_0)=x_0+\Omega,
\]
where $x_0\in\HH$ is any vector such that $Vx_0=z_0$ and 
\begin{equation*}
\Omega:=\setB{y\in\CV: (T^\#T+\lambda V^\#V)(x_0+y)=T^\#w_0+\lambda V^\#z_0\,\text{ for some $\lambda\in[\rho_-,\rho_+]$}}.
\end{equation*}
We show how the structures of $\Omega$ and $\mc{Z}(w_0,z_0)$ depend on the location of $\lambda$ in the interval $[\rho_-,\rho_+]$. The main result
of this section asserts that the set of solutions to Problem \ref{pb 1} is an affine manifold parallel to 
$N(T)\cap N(V)$ for every initial data $(w_0,z_0)$ belonging to an open and dense subset of the vector space $\KK\x\EE$.

As an application of the previous results, Section \ref{mixed} presents a generalization of the abstract mixed splines problem.

\section{Preliminaries}\label{preliminaries}

Along this work $\HH$ denotes a complex (separable) Hilbert space. If $\mc{K}$ is another Hilbert space then $\mc{L}(\HH, \KK)$ is the vector space of bounded linear operators from $\HH$ into $\KK$ and
$\mc{L}(\HH)=\mc{L}(\HH,\HH)$ stands for the algebra of bounded linear operators in $\HH$.

If $T\in \mc{L}(\HH, \KK)$ then $R(T)$ stands for the range of $T$ and $N(T)$ for its nullspace. The Moore-Penrose inverse
of an operator $T\in \mathcal{L}(\HH,\KK)$ is denoted by $T^\dag$.
Recall that $T^\dag\in\mc{L}(\KK,\HH)$ if and only if $T$ has closed range. 
For detailed expositions on the Moore-Penrose inverse, see \cite{BG, Nashed}. 

The reduced minimum modulus $\gamma(T)$ of an operator $T\in\mc{L}(\HH)$ is defined by
\[
\gamma(T)=\inf\set{\|Tx\|\,:\,\|x\|=1,\ x\in N(T)^\bot}.
\]
An operator $T\neq0$ has closed range if and only if $\gamma(T)>0$. In this case, $\gamma(T)=\|T^\dag\|^{-1}$.

An operator $A\in \mc{L}(\HH)$ is positive semidefinite if $\PI{Ax}{x}\geq 0$ for all $x\in\HH$; and it is positive definite if there exists
$\alpha>0$ such that $\PI{Ax}{x}\geq \alpha\|x\|^2$ for every $x\in\HH$. The cone of positive semidefinite operators is denoted by $\mc{L}(\HH)^+$.
We say that a selfadjoint operator $A\in\mc{L}(\HH)$ is indefinite if it is neither positive nor negative semidefinite, i.e. if there exist
$x_+,x_-\in\HH$ such that $\PI{Ax_+}{x_+}>0$ and $\PI{Ax_-}{x_-}<0$.

\subsection{Krein spaces}

In what follows we present the standard notation and some basic results on Krein spaces. For a complete exposition on the subject (and the proofs of the results below) see
\cite{Ando,Azizov,Bognar,Dritschel,Rovnyak}.

\medskip

An indefinite inner product space $(\mc{F}, \K{\cdot}{\cdot})$ is a (complex) vector space $\mc{F}$ endowed with a Hermitian sesquilinear form $\K{\cdot}{\cdot}: \mc{F}\x\mc{F} \ra \CC$.

A vector $x\in\mc{F}$ is {\it positive}, {\it negative}, or {\it neutral} if $\K{x}{x}>0$, $\K{x}{x}<0$, or $\K{x}{x}=0$, respectively.
Likewise, a subspace $\M$ of $\mc{F}$ is {\it positive} if every $x\in\M$, $x\neq0$ is a positive vector in $\mc{F}$; and it is
{\it nonnegative} if $\K{x}{x} \geq0$ for every $x\in\M$. Negative, nonpositive and neutral subspaces are defined mutatis mutandis.

If $\St$ is a subset of an indefinite inner product space $\mc{F}$, the \emph{orthogonal companion} to $\St$ is defined by 
\[
\St^{\ort}=\set{ x\in\mc{F} : \K{x}{s}=0 \; \text{for every $s\in\St$}}.
\]
It is easy to see that $\St^{\ort}$ is always a subspace of $\mc{F}$.

\begin{Def}
An indefinite inner product space $(\HH, \K{\cdot}{\cdot})$ is a \emph{Krein space} if it can be decomposed as a direct (orthogonal) sum of a Hilbert space and an anti Hilbert space, i.e. there
exist subspaces $\HH_\pm$ of $\HH$ such that $(\HH_+, \K{\cdot}{\cdot})$ and $(\HH_-, -\K{\cdot}{\cdot})$ are Hilbert spaces,
\begin{equation}\label{desc cano}
\HH=\HH_+ \dotplus \HH_-,
\end{equation}
and $\HH_+$ is orthogonal to $\HH_-$ with respect to the indefinite inner product. Sometimes we use the notation $\K{\cdot}{\cdot}_\HH$ instead of $\K{\cdot}{\cdot}$ to emphasize the Krein space considered.
\end{Def}

A pair of subspaces $\HH_\pm$ as in \eqref{desc cano} is called a \emph{fundamental decomposition} of $\HH$. Given a Krein space $\HH$ and a fundamental decomposition
$\HH=\HH_+\dotplus \HH_-$, the direct (orthogonal) sum of the Hilbert spaces $(\HH_+, \K{\cdot}{\cdot})$ and $(\HH_-, -\K{\cdot}{\cdot})$ is denoted
by $(\HH,\PI{\cdot}{\cdot})$.

If $\HH=\HH_+ \dotplus \HH_-$ and $\HH=\HH'_+ \dotplus \HH'_-$ are two different fundamental decompositions of $\HH$, the corresponding associated inner products $\PI{\cdot}{\cdot}$ and
$\PI{\cdot}{\cdot}'$ turn out to be equivalent on $\HH$. Therefore, the norm topology on $\HH$ does not depend on the chosen fundamental decomposition.

\medskip

A set $\M$ of a Krein space $(\HH,\K{\cdot}{\cdot})$ is {\it uniformly positive} if 
there exists $\alpha>0$ such that 
\[
\K{x}{x}\geq\alpha\|x\|^2\quad \text{ for every $x\in\M$},
\]
where $\|\cdot\|$ is the norm of any associated Hilbert space. Uniformly negative sets are defined mutatis mutandis.

If $(\HH,\K{\cdot}{\cdot}_\HH)$ and $(\KK,\K{\cdot}{\cdot}_\KK)$ are Krein spaces, $\mc{L}(\HH, \KK)$ stands for the vector space of linear transformations which are
bounded with respect to any of the associated Hilbert spaces $(\HH,\PI{\cdot}{\cdot}_\HH)$ and $(\KK,\PI{\cdot}{\cdot}_\KK)$. 
Given $T\in \mathcal{L}(\HH,\KK)$, the adjoint operator of $T$ (in the Krein spaces sense) is the unique operator $T^\#\in \mc{L}(\KK, \HH)$ such that
\[
\K{Tx}{y}_\KK=\K{x}{T^\#y}_\HH, \ \ \ \ x\in\HH,\,y\in\KK.
\]

We frequently use that if $T\in \mc{L}(\HH,\KK)$ and $\M$ is a closed subspace of $\KK$ then 
\begin{equation*}
T^\#(\M)^{\ort_\HH}= T^{-1}(\M^{\ort_\KK}).
\end{equation*}


\subsection{A version of Farkas' Lemma}

Let $(\HH,\PI{\cdot}{\cdot})$ be a Hilbert space, and $(\KK,\K{\cdot}{\cdot}_\KK)$, $(\EE,\K{\cdot}{\cdot}_\EE)$ be two Krein spaces.
Let $T\in\mc{L}(\HH,\KK)$ and $V\in\mc{L}(\HH,\EE)$. 
Recall that $\CV$ denotes the set of neutral vectors of the quadratic form associated to $V^\#V$:
\[
\CV=\setB{y\in\HH\,\,:\,\,\K{Vy}{Vy}=0}.
\]

If $V^\#V$ is a positive (or negative) semidefinite operator in $\HH$, then $\CV$ coincides with $N(V)$.
But, if $V^\#V$ is \emph{indefinite}, the set $\CV$ is strictly larger than $N(V)$.
From now on $V^\#V$ is assumed to be indefinite; i.e. neither positive nor negative semidefinite.

\medskip

The following result can be interpreted
as another manifestation of the S-Lemma (or Farkas' lemma), see \cite{Polik,Xia}. It first appeared in \cite{IKL82}. For
its proof, see Lemma 1.35 and Corollary 1.36 in \cite[Chapter 1, \S 1]{Azizov}. 
\begin{prop}\label{prop:asimov}
Given $T\in \mc{L}(\HH,\KK)$ and $V\in \mc{L}(\HH,\EE)$, the following conditions are equivalent:
\begin{enumerate}[label=\roman*)]
\item $T(\CV)$ is a nonnegative set of $\KK$;
\item there exists $\rho\in\RR$ such that $T^\#T+\rho V^\#V$ is positive semidefinite.
\end{enumerate}
\end{prop}

Let us also consider the subsets of $\HH$ where the quadratic form associated to $V^\#V$ takes positive and negative values:
\begin{equation*}
\PV:=\set{x\in\HH\,:\,\K{Vx}{Vx}>0}\quad\text{and}\quad\PVV:=\set{x\in\HH\,:\,\K{Vx}{Vx}<0}.
\end{equation*}

\begin{cor}\label{cor:intervalo}
If $T(\CV)$ is a nonnegative set of $\KK$, then
\begin{equation*}
\rho_-:=-\inf_{x\in\PV}\frac{\K{Tx}{Tx}}{\K{Vx}{Vx}}<+\infty\quad,\quad\rho_+:=-\sup_{x\in\PVV}\frac{\K{Tx}{Tx}}{\K{Vx}{Vx}}>-\infty,
\end{equation*}
and $\rho_-\leq\rho_+$. In this case, 
\[
T^\#T+\rho V^\#V\text{ is positive semidefinite if and only if }\rho\in[\rho_-,\rho_+].
\]
\end{cor}

\medskip

If $\rho_-\neq\rho_+$, the positive operators $T^\#T+\rho V^\#V$ with $\rho\in (\rho_-,\rho_+)$ share many properties.
We collect here some of the results from \cite{GZMMP21}, which are used along the paper.



\begin{lema}\label{lema:nucleo_abierto} Assume that $T(\CV)$ is a nonnegative set of $\KK$ and that $\rho_-\neq\rho_+$. Then 
\begin{align*}
N(T^\#T+\rho V^\#V) &= N(T^\#T)\cap N(V^\#V),\quad\text{for every}\quad\rho\in(\rho_-,\rho_+).
\end{align*}
\end{lema}

\begin{prop}\label{prop:rango_raices} Assume that $T(\CV)$ is a nonnegative set of $\KK$ and that $\rho_-\neq\rho_+$. Then
\[
R\big((T^\#T+\rho V^\#V)^{1/2}\big)=R\big((T^\#T+\rho' V^\#V)^{1/2}\big),\quad\quad\text{for every $\rho,\rho'\in(\rho_-,\rho_+)$.}
\]
Also, $R\big((T^\#T+\rho_\pm V^\#V)^{1/2}\big)\subseteq R\big((T^\#T+\rho V^\#V)^{1/2}\big)$, for every $\rho\in(\rho_-,\rho_+)$.
\end{prop}


\begin{prop}\label{prop:uniformemente} 
The following conditions are equivalent:
\begin{enumerate}[label=\roman*)]
\item there exists $\alpha>0$ such that $\K{Ty}{Ty}\geq\alpha\|y\|^2$ for every $y\in\CV$;
\item there exists $\rho\in\RR$ such that $T^\#T+\rho V^\#V$ is a positive definite operator.
\end{enumerate}
In this case, $\CT\cap\CV=\set{0}$.
\end{prop}

\section{Indefinite least squares problems with a quadratic constraint}\label{least squares}

From now on $(\HH,\PI{\cdot}{\cdot})$ denotes a Hilbert space, $(\KK,\K{\cdot}{\cdot}_\KK)$ and $(\EE,\K{\cdot}{\cdot}_\EE)$
denote Krein spaces; and  $T\in\mc{L}(\HH,\KK)$ has closed range and  $V\in \mc{L}(\HH,\EE)$ is surjective. 
The quadratically constrained ILSP under consideration is the following:

\begin{customprob}{1'}\label{pb 1'}
Given $x_0\in\HH$ and $w_0\in\KK$, analyze the existence of
\begin{equation*}
\min_{y\in\CV} \K{T(x_0 + y)-w_0}{T(x_0 + y)-w_0}_\KK,
\end{equation*}
and if the minimum exists, find the set of arguments at which it is attained.
\end{customprob}

Problem \ref{pb 1} is equivalent to
Problem 1'. In fact, Problem 1' with initial data $(w_0,x_0)$ is the same as Problem \ref{pb 1} with the initial data $(w_0,z_0)$
where $z_0:=Vx_0$. Conversely, Problem \ref{pb 1} with initial data $(w_0,z_0)$ can be rephrased as Problem 1' with initial data
$(w_0,x_0)$ where $x_0\in\HH$ is any vector such that $Vx_0=z_0$.

Moreover, the set of solutions to both problems is the same and we refer to them indistinctly as $\ZZ(w_0,z_0)$.

We begin by studying under which conditions the infimum among the values of the
objective function $x\mapsto\K{Tx-w_0}{Tx-w_0}$ over the set $x_0+\CV$ is finite.

\begin{prop}\label{prop infimo} 
Given $x_0\in\HH$ and $w_0\in\KK$, the following conditions are equivalent:
\begin{enumerate}[label=\roman*)]
\item $\displaystyle{\inf_{y\in\CV}\K{T(x_0+y)-w_0}{T(x_0+y)-w_0}>-\infty}$;\hspace*{\fill}\eqnum\label{ecu infimo}
\item there exists a constant $c\geq0$ such that
\begin{equation}\label{ecu desig K}
\modulo{\K{Tx_0-w_0}{Ty}}^2\leq c\K{Ty}{Ty},\quad\quad\textrm{for every $y\in\CV$}.
\end{equation}
\end{enumerate}
\end{prop}

\begin{proof} 
Suppose that $\inf_{y\in\CV}\K{T(x_0+y)-w_0}{T(x_0+y)-w_0}=k>-\infty$.\\ Then, for every $y\in\CV$,
\begin{equation}\label{ecu desig K 1}
\K{Ty}{Ty}+2\real\K{Tx_0-w_0}{Ty}+\K{Tx_0-w_0}{Tx_0-w_0}-k\geq0.
\end{equation} 
Replacing $y$ by $ty$ for a fixed $y\in\CV$ and $t\in\mathbb{R}$, \eqref{ecu desig K 1} gives
\begin{equation}\label{parabola real K'}
a t^2 + bt + c\geq 0 \quad\quad\textrm{for every $t\in\RR$},
\end{equation}
where $a=\K{Ty}{Ty}$, $b=2\real\K{Tx_0-w_0}{Ty}$ and $c=\K{Tx_0-w_0}{Tx_0-w_0}-k\geq0$.
But \eqref{parabola real K'} holds if and only if $a\geq 0$ and $b^2-4ac\leq 0$, i.e.
\begin{equation*}
\big(\real\K{Tx_0-w_0}{Ty}\big)^2\leq c\K{Ty}{Ty},\quad\quad\textrm{for every $y\in\CV$}.
\end{equation*}
Now, if $\K{Tx_0-w_0}{Ty}=e^{i\theta}\modulo{\K{Tx_0-w_0}{Ty}}$, with $\theta\in[0,2\pi)$, set $v:=e^{i\theta}y\in\CV$,
then $\K{Tv}{Tv}=\K{Ty}{Ty}$ and $\real \K{Tx_0-w_0}{Tv}=\modulo{\K{Tx_0-w_0}{Ty}}$. Therefore, 
\begin{equation*}
\modulo{\K{Tx_0-w_0}{Ty}}^2\leq c\K{Ty}{Ty},\quad\quad\textrm{for every $y\in\CV$}.
\end{equation*}

Conversely, let $c\geq0$ be such that \eqref{ecu desig K} holds. Then $\K{Ty}{Ty}\geq0$ for every $y\in\CV$ and
\begin{equation*}
\big(\real\K{Tx_0-w_0}{Ty}\big)^2\leq\modulo{\K{Tx_0-w_0}{Ty}}^2\leq c\K{Ty}{Ty}.
\end{equation*}
For an arbitrary (fixed) vector $y\in\CV$ define $a$ and $b$ as above. Therefore, $a\geq0$, $b^2-4ac\leq0$, and \eqref{parabola real K'}
follows. Or equivalently, 
\begin{equation*}
\K{T(x_0+ty)-w_0}{T(x_0+ty)-w_0}\geq\K{Tx_0}{Tx_0}-c,
\end{equation*}
where $y\in\CV$ and $t\in\mathbb{R}$. Since $y\in\CV$ is arbitrary, \eqref{ecu infimo} holds.
\end{proof}








In view of Proposition \ref{prop infimo}, we assume that the following hypotheses hold for the
rest of this section.

\begin{hyp}\label{T del cono}
$T^\#T$ and $V^\#V$ are indefinite operators on $\HH$ and 
\[
T(\CV)\ \text{ is a nonnegative set of $\KK$}.
\]
\end{hyp}

If $T^\#T$ is a semidefinite operator then Problem \ref{pb 1'} turns out to be a
least-squares problem with a quadratic constraint instead of an indefinite least-squares problem, and the results below also
hold in this case with some minor adjustments. 



\medskip

The existence of solutions to Problem \ref{pb 1'} is equivalent to the existence of a vector $y_0\in\CV$ such that 
$c:=\K{Ty_0}{Ty_0}$ satisfies \eqref{ecu desig K}.
This is expressed in the next proposition; the proof follows the lines of the proof of
\cite[Proposition 3.1]{JOTA}.

\begin{prop}\label{prop cond 1} Given $(w_0,z_0)\in\KK\x\EE$, let $x_0\in\HH$ be such that $Vx_0=z_0$.
Then, $\ZZ(w_0,z_0)\neq \varnothing$
if and only if there exists $y_0\in\CV$ such that 
	\begin{equation}\label{ecu desig modulo}
	\modulo{\K{Tx_0-w_0}{Ty}}^2 \leq \K{Ty_0}{Ty_0}\K{Ty}{Ty},\quad\quad\textrm{for every $y\in\CV$},
	\end{equation}
with equality when $y=y_0$.

In this case, $x_0+y_0\in\ZZ(w_0,z_0)$ if and only if $y_0\in\CV$ satisfies \eqref{ecu desig modulo} and 
\begin{equation}\label{eq normal}
\K{T(x_0+y_0)-w_0}{Ty_0}=0.
\end{equation}
\end{prop}

Another characterization of the existence of solutions to Problem \ref{pb 1'} can be given by means of a normal equation.
We study first the case of
solutions $\widetilde{x}$ satisfying the stronger constraint $V\wt{x}=z_0$.

\begin{lema}\label{lema:caso_nucleo} Given $(w_0,z_0)\in\KK\x\EE$,
let $\widetilde{x}\in\HH$ such that $V\widetilde{x}=z_0$. Then, 
\[
\widetilde{x}\in \ZZ(w_0,z_0) \ \text{ if and only if } \ T^\#T\widetilde{x}=T^\#w_0.
\]
In this case, $\ZZ(w_0,z_0)=\widetilde{x}+\CT\cap\CV$.
\end{lema}

\begin{proof}
Suppose that $\widetilde{x}\in \ZZ(w_0,z_0)$, and let $y\in\CV$.
By Proposition \ref{prop cond 1},
\begin{equation*}
\modulo{\K{T\widetilde{x}-w_0}{Ty}}^2\leq\K{T\widetilde{y}_0}{T\widetilde{y}_0}\K{Ty}{Ty},\quad\quad\text{for all $y\in\CV$},
\end{equation*}
where $\widetilde{y}_0$ is any vector in $\CV$ such that
\[
\min_{y\in\CV}\ \K{T(\widetilde{x}+y)-w_0}{T(\widetilde{x}+y)-w_0}=\K{T(\widetilde{x}+\widetilde{y}_0)-w_0}{T(\widetilde{x}+\widetilde{y}_0)-w_0}.
\]
But since $\widetilde{x}\in\ZZ(w_0,z_0)$, we can take $\widetilde{y}_0=0$
and thus $\K{T\widetilde{x}-w_0}{Ty}=0$ for all $y\in\CV$. Hence, $T^\#(T\widetilde{x}-w_0)\in\CV^\bot=\set{0}$.

Conversely, assume that $T^\#T\widetilde{x}=T^\#w_0$. Let $y\in\CV$, then
\begin{align*}
& \hspace{-36pt}\K{T(\widetilde{x}+y)-w_0}{T(\widetilde{x}+y)-w_0} \\&=\K{T\widetilde{x}-w_0}{T\widetilde{x}-w_0}+\K{Ty}{Ty}+2\,\real\K{T\widetilde{x}-w_0}{Ty}\nonumber\\
&=\K{T\widetilde{x}-w_0}{T\widetilde{x}-w_0}+\K{Ty}{Ty}\nonumber\\
&\geq\K{T\widetilde{x}-w_0}{T\widetilde{x}-w_0},
\end{align*}
because $T(\CV)$ is a nonnegative set. Then $\widetilde{x}\in \ZZ(w_0,z_0)$. Moreover, the minimum
is attained if and only if $y\in\CT\cap\CV$.
\end{proof}

The following theorem establishes the normal equation that characterizes the solutions to Problem \ref{pb 1} in the general case. 
According to Hypothesis \ref{T del cono}, the parameters $\rho_\pm$ introduced in Corollary \ref{cor:intervalo}
are well-defined.

\begin{thm}\label{teo:derivada} 
Given $(w_0,z_0)\in\KK\x\EE$, let $\widetilde{x}\in\HH$. Then, $\widetilde{x}\in \ZZ(w_0,z_0)$ if and only if there exists $\lambda\in[\rho_-,\rho_+]$ such that
\begin{equation}\label{eq:eq_normal}
(T^\#T+\lambda V^\#V)\widetilde{x}=T^\#w_0+\lambda V^\#z_0,
\end{equation}
and
\[
\K{V\widetilde{x}-z_0}{V\widetilde{x}-z_0}=0.
\]
\end{thm}

\begin{proof}
Consider the function $F:\HH\ra\RR$ given by $F(x)=\K{Tx-w_0}{Tx-w_0}$. This function is Fr\'echet differentiable at every $x\in\HH$ and its Fr\'echet derivative at $x$ is given by:
\[
DF(x)\ \dx =2\real(\K{Tx-w_0}{T\dx}), \quad \dx\in\HH.
\]
Indeed, given $x\in\HH$,
\begin{eqnarray*}
& & \hspace{-36pt} \frac{\left| F(x + \dx) - F(x) - DF(x)\dx \right|}{\|\dx\|} \\
&=&  \frac{\left| 2\real(\K{Tx-w_0}{T\dx}) + \K{T\dx}{T\dx} - 2\real(\K{Tx-w_0}{T\dx}) \right|}{\|\dx\|} \\ &=& \frac{\left| \K{T\dx}{T\dx} \right|}{\|\dx\|} \leq \|T\|^2\|\dx\|\ra 0,
\end{eqnarray*}
as $\|\dx\|\ra 0$.
Analogously, the function $G:\HH\ra\RR$ given by $G(x)=\K{Vx-z_0}{Vx-z_0}$ is Fr\'echet differentiable at every $x\in\HH$ and its Fr\'echet derivative at $x$ is given by:
\[
DG(x)\ \dx =2\real(\K{Vx-z_0}{V\dx}), \quad \dx\in\HH.
\]
In the same fashion, the second order Fréchet derivatives at $x\in\HH$ are given by
\[
D^2F(x)(\Delta x_1,\Delta x_2)=2\real(\K{T\Delta x_1}{T\Delta x_2}),\quad \Delta x_1,\Delta x_2\in\HH,
\]
\[
D^2G(x)(\Delta x_1,\Delta x_2)=2\real(\K{V\Delta x_1}{V\Delta x_2}),\quad\Delta x_1,\Delta x_2\in\HH.
\]

Now, assume that $\widetilde{x}\in \ZZ(w_0,z_0)$. If $V\widetilde{x}=z_0$, then the result follows from Lemma \ref{lema:caso_nucleo},
choosing an arbitrary $\lambda\in[\rho_-,\rho_+]$. On the other hand, if $V\widetilde{x}\neq z_0$ then
$DG(\widetilde{x})\neq0$. Hence, by \cite[\S 7.7 Thm. 2]{Luenberger} there exists $\lambda\in\RR$ such that $DF(\widetilde{x}) + \lambda DG(\widetilde{x}) = 0$, i.e.	
\begin{align*}
\real(\K{T\widetilde{x}-w_0}{T\dx} + \lambda \K{V\widetilde{x}-z_0}{V\dx})=0,\quad\text{for every $\Delta x\in\HH$}.
\end{align*}
Replacing $\dx$ by $-i\dx$, the imaginary part is also zero. Thus,
\begin{align*}
\K{T\widetilde{x}-w_0}{T\dx} + \lambda \K{V\widetilde{x}-z_0}{V\dx}=0, \quad \text{for every $\dx\in\HH$}.
\end{align*}
Therefore,
\[
(T^\#T+\lambda V^\#V)\widetilde{x}=T^\#w_0+\lambda V^\#z_0.
\]
Moreover, by \cite[Prop. 2.4.19]{Mardsen}, for every $\Delta x\in\HH$,
\begin{align*}
0\leq D^2\parentesis{F+\lambda G}(\widetilde{x})\cdot(\Delta x,\Delta x)&=2\real\K{T\Delta x}{T\Delta x}+\lambda2\real\K{V\Delta x}{V\Delta x} \\
&=2\PI{(T^\#T+\lambda V^\#V)\Delta x}{\Delta x}.
\end{align*}
Hence $T^\#T+\lambda V^\#V\in\mc{L}(\HH)^+$, or equivalently, $\lambda\in[\rho_-,\rho_+]$ (see Corollary \ref{cor:intervalo}).

\smallskip

Conversely, assume that $\K{V\widetilde{x}-z_0}{V\widetilde{x}-z_0}=0$ and that there exists $\lambda\in[\rho_-,\rho_+]$ such that $(T^\#T+\lambda V^\#V)\widetilde{x}=T^\#w_0+\lambda V^\#z_0$.
Given $x_0\in\HH$ such that $Vx_0=z_0$ there exists $y_0\in\CV$ such that
$\widetilde{x}=x_0+y_0$.
Then, 
\begin{equation}\label{eq:eqnormal_2}
(T^\#T+\lambda V^\#V)y_0=-T^\#(Tx_0-w_0),
\end{equation}
and $\K{Tx_0-w_0}{Ty_0}=-\PI{(T^\#T+\lambda V^\#V)y_0}{y_0}=-\K{Ty_0}{Ty_0}$. Hence, $\widetilde{x}=x_0+y_0$ satisfies \eqref{eq normal}.

Since $T^\#T+\lambda V^\#V$ is positive semidefinite for $\lambda\in[\rho_-,\rho_+]$,
\begin{align*}
\modulo{\K{Tx_0-w_0}{Ty}}^2 &=\modulo{\PI{-(T^\#T+\lambda V^\#V)y_0}{y}}^2 \\ &\leq\PI{(T^\#T+\lambda V^\#V)y_0}{y_0}\PI{(T^\#T+\lambda V^\#V)y}{y}\\
&=\K{Ty_0}{Ty_0}\K{Ty}{Ty},
\end{align*}
because $\PI{(T^\#T+\lambda V^\#V)y}{y}=\K{Ty}{Ty}$, for all $y\in\CV$.
Then, the result follows from Proposition \ref{prop cond 1}.
\end{proof}

\medskip

The next example shows how $\ZZ(w_0,z_0)$ depends on the initial data $(w_0,z_0)$, even in the situation in which the spectral decompositions determined by $T^\#T$ and $V^\#V$ are very simple.

\begin{example}\label{ejemplo indef} 
Assume that $\HH$ is decomposed as $\HH=\HH_1\oplus\HH_2\oplus \HH_3$  and consider operators $T\in\mc{L}(\HH,\KK)$ and $V\in\mc{L}(\HH,\EE)$ such that $T^\#T$ and $V^\#V$ can be represented by
\[
T^\#T=\begin{bmatrix}
 I	&	0			&	0\\
0			&	-\frac{1}{2} I			&	0\\
0			&	0			&	I
\end{bmatrix}
\qquad \text{and} \qquad
V^\#V=\begin{bmatrix}
4 I	&	0			&	0\\
0			&	I			&	0\\
0			&	0			&	-I
\end{bmatrix}.
\]
The operator $T^\#T+\rho V^\#V$ is positive semidefinite if and only if $1+4\rho\geq 0$, $\rho-\frac{1}{2}\geq 0$ and $1-\rho\geq 0$. Hence, it is readily seen that
\[
\rho_-=\frac{1}{2}\quad\text{and}\quad\rho_+=1.
\]
In the following we show that $\ZZ(w_0,z_0)\neq\varnothing$ for every $(w_0,z_0)\in\KK\x\EE$, and we describe $\ZZ(w_0,z_0)$ in each case.
Since $\K{Vx}{Vx}=4\|x_1\|^2+\|x_2\|^2-\|x_3\|^2$,
the set $\CV$ can be described as
\begin{equation*}
\CV=\set{y_1+y_2+\parentesis{4\|y_1\|^2+\|y_2\|^2}^{1/2}y_3\,:\,y_1\in\HH_1,\,y_2\in\HH_2,\,y_3\in\St_{3}}, 
\end{equation*}
where $\St_{i}$ stands for the unit sphere in $\HH_i$ for $i=1,2,3$.

Given $(w_0,z_0)\in\KK\x\EE$, let $x_0\in\HH$ be such that $Vx_0=z_0$. If $y\in\CV$, Theorem \ref{teo:derivada} assures that $\wt{x}=x_0 + y\in \ZZ(w_0,z_0)$ if and only if there exists $\lambda\in [\frac{1}{2}, 1]$ such that \eqref{eq:eq_normal} holds, or equivalently, if
\begin{equation}\label{casi casi}
(T^\#T + \lambda V^\#V)y=-T^\#(Tx_0-w_0).
\end{equation}
Writing $-T^\#(Tx_0-w_0)=x_1+x_2+x_3$,
with $x_i\in\HH_i$, and decomposing $y\in\CV$ as
\[
y=y_1+y_2+(4\|y_1\|^2+\|y_2\|^2)^{1/2}y_3,
\]
with $y_1\in\HH_1$, $y_2\in\HH_2$, and $y_3\in \mc{S}_{3}$, \eqref{casi casi} reads as
\begin{align}
(1+4\lambda)y_1 &= x_1, \label{eq:super_1}\\ 
(\lambda-\tfrac{1}{2})y_2 &= x_2, \label{eq:super_2}\\
(1-\lambda)\parentesis{4\|y_1\|^2+\|y_2\|^2}^{1/2}y_3 &=x_3. \label{eq:super_3}
\end{align}
Since $\lambda\in[\frac{1}{2},1]$, \eqref{eq:super_1} says that $y_1=\frac{1}{1+4\lambda}x_1$. If $x_1+x_2+x_3=0$, it is easy to see that $\ZZ(w_0,z_0)=\set{x_0}$. In the following, we study the situations where this is not the case.

\begin{itemize}
\item \textbf{Case 1}: $x_3=0$.

Since $y_3\neq0$, \eqref{eq:super_3} yields $\lambda=1$. Moreover,
\[
\ZZ(w_0,z_0)=x_0+\tfrac{1}{5} x_1+2x_2+\parentesis{\tfrac{4}{25} \|x_1\|^2+4\|x_2\|^2}^{1/2}\St_{3}.
\]
\item \textbf{Case 2}: $x_3\neq 0$ and $x_2\neq0$.

In this case, \eqref{eq:super_2} and \eqref{eq:super_3} yield $\lambda\in(\frac{1}{2},1)$. Therefore,
\[
\ZZ(w_0,z_0)=\set{x_0+\tfrac{1}{1+4\lambda}x_1+\tfrac{1}{\lambda-\tfrac{1}{2}}x_2+\tfrac{1}{1-\lambda}x_3}.
\]

\item \textbf{Case 3}: $x_3\neq 0$ and $x_2=0$.

%
%
In this case, two different situations have to be considered. Indeed, \eqref{eq:super_2} implies that either $y_2=0$ or $\lambda=\frac{1}{2}$.
Also, \eqref{eq:super_3} says that
\[
\frac{4}{(1+4\la)^2}\|x_1\|^2 + \|y_2\|^2=\frac{1}{(1-\la)^2}\|x_3\|^2.
\]
Denoting $\gamma=\frac{\|x_1\|}{\|x_3\|}$, we can distinguish between two different cases: 
\begin{enumerate}[label=\roman*)]
\item if $\gamma> 3$ then $\la:=\frac{2\gamma -1}{2\gamma+4}$ is contained in the interval $(\tfrac{1}{2},1)$, hence $y_2=0$ and
\[
\ZZ(w_0,z_0)=\set{x_0+\tfrac{1}{1+4\la}x_1+\tfrac{1}{1-\la}x_3};
\]
\item if $\gamma\leq3$ then $\la=\frac{1}{2}$ and
\[
\ZZ(w_0,z_0)= x_0+\tfrac{1}{3}x_1 + 2\|x_3\|\big(1-\tfrac{\gamma^2}{9}\big)^{1/2}\St_{2}+ 2x_3.
\]
\end{enumerate}

\end{itemize}

\end{example}

\begin{obs}
The above example can be easily generalized, replacing the constants $4$ and $\frac{1}{2}$ appearing in the block matrix representations of $V^\#V$ and $T^\#T$ by arbitrary reals $\alpha$ and $\beta$ such that $\alpha>1$ and $0<\beta<1$, respectively. 

Given $\alpha$ and $\beta$ such that $\alpha>1$ and $0<\beta<1$, the parameter $\lambda$ varies between $\rho_-=\beta$ and $\rho_+=1$. Then, Case 3 splits into two according to $\tfrac{\|x_1\|}{\|x_3\|}\geq \gamma_{\alpha}$ or  $\tfrac{\|x_1\|}{\|x_3\|}< \gamma_{\alpha}$, where $\gamma_{\alpha}$ is a constant depending of $\alpha$. If $\tfrac{\|x_1\|}{\|x_3\|}\geq \gamma_{\alpha}$ then $\lambda\in (\beta,1)$, and if $\tfrac{\|x_1\|}{\|x_3\|}< \gamma_{\alpha}$ then $\lambda=\beta$.
\end{obs}

In the example above, Problem \ref{pb 1'} admits solution for every $(w_0,z_0)\in\KK\x\EE$, mainly due to the invertibility
of the operator $T^\#T + \rho V^\#V$ for $\rho\in (\rho_-,\rho_+)$. The next section presents necessary and sufficient conditions for 
the existence of solutions to Problem \ref{pb 1} for arbitrary initial data.

\section{Necessary and sufficient conditions for the existence of solutions for arbitrary initial data}\label{condiciones sufi}

The aim of this section is to characterize under which conditions Problem \ref{pb 1} admits a solution for every
$(w_0,z_0)\in\KK\x\EE$. To do so we suppose that $N(T)\cap N(V)=\set{0}$. Later on we express the results for the general case.
We assume the following:
\begin{hyp2}\label{hipo sec 4_2} 
$T^\#T$ and $V^\#V$ are indefinite operators on $\HH$, such that
\[
N(T)\cap N(V)=\set{0}.
\]
\end{hyp2}

\medskip

We first show some necessary conditions.
 
\begin{lema}\label{lema:condiciones_necesarias}
Assume that $\ZZ(w,z)\neq\varnothing$ for every $(w,z)\in\KK\x\EE$. Then:
\begin{enumerate}[label=\roman*)]
\item $\rho_-\neq\rho_+$;
\item $N(T^\#T)\cap N(V)=\set{0}$;
\item $\HH=N(T^\#T)^\bot+N(V)^\bot$.
\end{enumerate}
\end{lema}

\begin{proof}
\noi {\it i)} Assume that $\rho_-=\rho_+=\rho$. Then, by Theorem \ref{teo:derivada}, for any $(w,z)\in\KK\x\EE$ there exists $\widetilde{x}\in\HH$ such that
\[
(T^\#T+\rho V^\#V)\widetilde{x}=T^\#w+\rho V^\#z.
\]
Since $(w,z)$ is arbitrary, $R(T^\#)\subseteq R(T^\#T+\rho V^\#V)$. But this is a contradiction to \cite[Thm. 4.17]{GZMMP21}.
Therefore, $\rho_-\neq\rho_+$.


\smallskip

\noi {\it ii)} By Lemma \ref{lema:nucleo_abierto}, $\rho_-\neq\rho_+$ implies that
$\CT\cap\CV=N(T^\#T)\cap N(V)$. Given $x_0\in\HH$ and $w_0\in\KK$, by Proposition \ref{prop cond 1} 
there exists $y_0\in\CV$ such that
$$|\K{Tx_0-w_0}{Ty}|^2\leq\K{Ty_0}{Ty_0}\K{Ty}{Ty},\quad\text{for all $y\in\CV$.}$$
Hence, $Tx_0-w_0\in T(\CT\cap\CV)^\ort$. Since $x_0$ and $w_0$ are arbitrary,
$$T\big(N(T^\#T)\cap N(V)\big)^\ort=T(\CT\cap\CV)^\ort=\KK,$$
and thus $N(T^\#T)\cap N(V)\subseteq N(T)$. Consequently, $N(T^\#T)\cap N(V)=N(T)\cap N(V)=\set{0}$.

\smallskip

\noi {\it iii)} By item {\it ii}, we only need to show that $N(T^\#T)+ N(V)$ is closed. 
Assume that $(x_n)_{n\in\NN}$ is a sequence in $N(T^\#T)$ and $(u_n)_{n\in\NN}$ is a sequence in $N(V)$ such that
$x_n + u_n\ra x_0\in\overline{N(T^\#T)+N(V)}$. Since $\ZZ(0,Vx_0)\neq\varnothing$, by Theorem \ref{teo:derivada} there exist $\la\in [\rho_-,\rho_+]$ and $y_0\in\CV$ such that
\begin{equation}\label{eq:una_eq_normal}
(T^\#T + \la V^\#V)y_0=-T^\#Tx_0.
\end{equation}

In what follows we prove that in this case, $y_0\in N(V)$; then, by \eqref{eq:una_eq_normal} $x_0+y_0\in N(T^\#T)$, or equivalently, $x_0\in N(T^\#T)+N(V)$.

On the one hand, since $x_n\in N(T^\#T)$,
\begin{equation}\label{eq:coso}
T^\#Tu_n=T^\#T(u_n+x_n)\to T^\#Tx_0=-(T^\#T+\lambda V^\#V)y_0.
\end{equation}
On the other hand, since $u_n\in N(V)$ and $y_0\in\CV$, $\PI{V^\#V(u_n+y_0)}{u_n+y_0}=0$ for every $n\in\NN$.
Then for any $\rho\in(\rho_-,\rho_+)$, $\rho\neq\lambda$, it holds that
\[
(T^\#T+\rho V^\#V)^{1/2}(u_n+y_0)\to0.
\]

In fact,
\begin{align*}
\|(T^\#T+\rho V^\#V)^{1/2}(u_n+y_0)\|^2&=\PI{T^\#T(u_n+y_0)}{u_n+y_0}\\
&=\PI{T^\#T(u_n+y_0)}{u_n}+\PI{T^\#T(u_n+y_0)}{y_0}.
\end{align*}
Also, since $y_0\in\CV$,
\[
\PI{T^\#T(u_n+y_0)}{y_0}=\PI{T^\#Tu_n+(T^\#T+\lambda V^\#V)y_0}{y_0}\to 0,
\] 
and
\begin{align*}
\PI{T^\#T(u_n+y_0)}{u_n} &=\PI{T^\#T(x_n+u_n+y_0)}{x_n+u_n} \\ &\to\PI{T^\#T(x_0+y_0)}{x_0}=0,
\end{align*}
because $T^\#T(x_0+y_0)=-\lambda V^\#Vy_0\in N(T^\#T)^\bot\cap N(V)^\bot$, see \eqref{eq:una_eq_normal}.

Therefore $(T^\#T+\rho V^\#V)(u_n+y_0)\to 0$, or equivalently,
$T^\#Tu_n\to -(T^\#T+\rho V^\#V)y_0$, for every $\rho\in(\rho_-,\rho_+)$ 
But, by \eqref{eq:coso},  $T^\#Tu_n \ra -(T^\#T + \la V^\#V)y_0$. Hence, $(T^\#T + \rho V^\#V)y_0=(T^\#T + \la V^\#V)y_0$.
So that $V^\#Vy_0=0$, and thus $y_0\in N(V)$ and
\[
T^\#Ty_0=(T^\#T + \la V^\#V)y_0=-T^\#Tx_0.
\]
Then $x_0+y_0\in N(T^\#T)$, or $x_0\in N(T^\#T)+N(V)$ as claimed.
\end{proof}

Lemma \ref{lema:condiciones_necesarias} allows us to prove the following necessary condition
for Problem \ref{pb 1} admitting a solution for every initial data point.

\begin{prop}\label{prop:vuelta} Assume that $\ZZ(w,z)\neq\varnothing$ for every $(w,z)\in\KK\x\EE$. Then
there exists $\alpha>0$ such that $\K{Ty}{Ty}\geq\alpha\|y\|^2$
for every $y\in\CV$
\end{prop}

\begin{proof}
By Lemma \ref{lema:condiciones_necesarias} $\rho_-\neq\rho_+$, $N(T^\#T)\cap N(V)=\set{0}$ and
$\HH=N(T^\#T)^\bot+N(V)^\bot\subseteq N(T)^\bot+N(V)^\bot$.
Given $\rho'\in(\rho_-,\rho_+)$, we claim that $N(T)^\bot\subseteq R(T^\#T+\rho' V^\#V)^{1/2}$. In fact, let
$x_0\in N(T)^\bot=R(T^\#)$, and let $w_0\in\KK$ be such that $T^\#w_0=x_0$. Since $\ZZ(w_0,0)\neq\varnothing$, by Theorem \ref{teo:derivada} there
exist $\la\in [\rho_-,\rho_+]$ and $y_0\in\CV$ such that
\begin{equation*}
(T^\#T + \la V^\#V)y_0=T^\#w_0.
\end{equation*}
By Proposition \ref{prop:rango_raices},
\[
x_0=T^\#w_0=(T^\#T+\lambda V^\#V)y_0\subseteq R\big((T^\#T+\rho' V^\#V)^{1/2}\big).
\]
Since $x_0$ is arbitrary, we have that $N(T)^\bot\subseteq R\big((T^\#T+\rho' V^\#V)^{1/2}\big)$. Using this fact, it holds that
\[
N(V)^\bot=R(V^\#V)\subseteq R(T^\#)+R(T^\#T+\rho' V^\#V)\subseteq R\big((T^\#T+\rho' V^\#V)^{1/2}\big),
\]
which implies
\[
\HH=N(T)^\bot+N(V)^\bot=R\big((T^\#T+\rho' V^\#V)^{1/2}\big),
\]
and thus $\HH=R(T^\#T+\rho' V^\#V)$, see \cite{Douglas}. Hence, $T^\#T+\rho' V^\#V$ is a positive definite operator,
or equivalently, by Proposition \ref{prop:uniformemente}, there exists $\alpha>0$ such that $\K{Ty}{Ty}\geq\alpha\|y\|^2$
for every $y\in\CV$.
\end{proof}

Next we establish the conditions that guarantee the existence of solutions for arbitrary initial data.
For the rest of this section we assume the following hypothesis:

\begin{hyp2}\label{hipo sec 4} 
Given $T\in\mc{L}(\HH,\KK)$ and $V\in \mc{L}(\HH,\EE)$ such that $T^\#T$ and $V^\#V$ are indefinite operators on $\HH$, assume
that there exists $\alpha>0$ such that
\[
\K{Ty}{Ty}\geq\alpha\|y\|^2,\quad\quad\text{for every $y\in\CV$.}
\]
\end{hyp2}

This implies that $\CT\cap\CV=\set{0}$ (which in turn implies the condition in Hypothesis \ref{hipo sec 4_2}).
Later on, we modify Hypothesis \ref{hipo sec 4} appropriately to express the results for the general case.

\medskip

Hypothesis \ref{hipo sec 4} is not sufficient to ensure the existence of solutions
for every initial data point, as the next example shows.

\begin{example}
Assume that $\HH=\KK=\EE=\ell_2(\NN)$, and consider the indefinite inner products
$$\K{x}{y}_\KK=x_1\overline{y}_1-\sum_{k\geq2}x_k\overline{y}_k,\quad x=(x_k)_{k\in\NN},\,y=(y_k)_{k\in\NN}\in  \ell_2(\NN),$$
$$\K{x}{y}_\EE=-x_1\overline{y}_1+\sum_{k\geq2}x_k\overline{y}_k,\quad x=(x_k)_{k\in\NN},\,y=(y_k)_{k\in\NN}\in  \ell_2(\NN).$$
Then, $(\KK,\K{\cdot}{\cdot}_\KK)$ and $(\EE,\K{\cdot}{\cdot}_\EE)$ are Krein spaces.

If $\{e_k\}_{k\in\NN}$ denotes the standard canonical basis of $ \ell_2(\NN)$, and $\alpha>\beta>0$, consider the linear operators
$T:\HH\ra\KK$ and $V:\HH\ra\EE$ given by
\begin{alignat*}{3}
& Te_1=\alpha e_1, \quad && Te_k=e_{k}\,\text{ if }k\geq2,\\
& Ve_1=\beta e_1, \quad &&Ve_k=(1+\tfrac{1}{k})^{1/2}e_{k}\,\text{ if }k\geq2.
\end{alignat*}
Both $T$ and $V$ are trivially surjective, $N(T)=N(V)=\set{0}$, and a few calculations show that, for
$x=(x_k)_{k\in\NN}=\sum_{k\geq1}x_ke_k\in  \ell_2(\NN)$,
\begin{align*}
T^\#Tx&=\alpha^2x_1e_1-\sum_{k\geq2}x_ke_k,\\
V^\#Vx&=-\beta^2x_1e_1+\sum_{k\geq2}x_k(1+\tfrac{1}{k})e_k.
\end{align*}
Hence, $T^\#T$ and $V^\#V$ are two indefinite operators acting on $\HH$. Moreover,
$T^\#T+\rho V^\#V$ is positive semidefinite if and only if $1\leq\rho\leq\tfrac{\alpha^2}{\beta^2}$, i.e.
$\rho_-=1$ and $\rho_+=\tfrac{\alpha^2}{\beta^2}$. Also,
\begin{displaymath}
\gamma(T^\#T+\rho V^\#V)=\left\{\begin{array}{ll}
\rho-1						&,\quad 1\leq\rho<\tfrac{\alpha^2+1}{\beta^2+1},\\
\alpha^2-\rho\beta^2	&,\quad \tfrac{\alpha^2+1}{\beta^2+1}\leq\rho<\tfrac{\alpha^2}{\beta^2},\\
\tfrac{\alpha^2}{\beta^2}-1						&,\quad \rho=\tfrac{\alpha^2}{\beta^2}.
\end{array}\right.
\end{displaymath}
Then, $R(T^\#T+\rho V^\#V)$ is closed for every $\rho\in(1,\tfrac{\alpha^2}{\beta^2}]$. Given $\rho\in(1,\tfrac{\alpha^2}{\beta^2}]$,
\[
N(T^\#T+\rho V^\#V)=\set{0}=N(T)\cap N(V),
\]
and $T^\#T+\rho V^\#V$ is a positive definite operator, or equivalently
there exists $\gamma>0$ such that $\K{Ty}{Ty}\geq\gamma\|y\|^2$ for every
$y\in\CV$.

However, Problem \ref{pb 1} does not admit solutions for every
$(w_0,z_0)\in\KK\x\EE$. In fact, consider the vector $(0,Ve_1)=(0,\beta e_1)\in\KK\x\EE$. 
By Theorem \ref{teo:derivada}, Problem \ref{pb 1} admits a solution for $(0,\beta e_1)$ if and only if there exist
$\lambda\in[1,\tfrac{\alpha^2}{\beta^2}]$ and $y\in\CV$ such that
\[
(T^\#T+\lambda V^\#V)(e_1+y)=\lambda V^\#Ve_1,
\]
or equivalently,
\begin{equation}\label{eq:normal_ejemplo}
(T^\#T+\lambda V^\#V)y=-T^\#Te_1.
\end{equation}
On the one hand, note that $y=(y_k)_{k\in\NN}\in\CV$ if and only if
\[
\sum_{k\geq2}(1+\tfrac{1}{k})|y_k|^2=\beta^2|y_1|^2.
\]
On the other hand, \eqref{eq:normal_ejemplo} is equivalent to
\begin{align*}
(\alpha^2-\lambda\beta^2)y_1&=-\alpha^2,\\
\big[\lambda(1+\tfrac{1}{k})-1\big]y_k&=0,\quad\text{for $k\geq2$.}
\end{align*}
In this case there is no $y\in\CV$ satisfying \eqref{eq:normal_ejemplo}
because the above equations imply
\[
0=\sum_{k\geq2}(1+\tfrac{1}{k})|y_k|^2=\beta^2|y_1|^2,
\]
and thus $0=-\alpha^2$, leading to a contradiction.
\end{example}

\medskip

Hypothesis \ref{hipo sec 4} allows us to study a simpler equivalent problem, because in this
case the operator pencil $P(\lambda)=T^\#T+\lambda V^\#V$ is regular: by Proposition \ref{prop:uniformemente}
and \cite[Cor. 4.14]{GZMMP21}, $T^\#T+\rho V^\#V$ is positive definite for every $\rho\in(\rho_-,\rho_+)$.
Let us fix $\rho=\frac{\rho_-+\rho_+}{2}$ for convenience, and define the following indefinite inner product on $\KK\x \EE$:
\begin{equation}\label{met indef para KxE}
\K{(w,z)}{(w',z')}_{\rho}=\K{w}{w'}_\KK + \rho\K{z}{z'}_\EE,\quad\quad \textrm{$w,w'\in \KK$ and $z,z'\in\EE$}.
\end{equation}
It is easy to see that $(\KK\x\EE,\K{\cdot}{\cdot}_\rho)$ is a Krein space.
Define the operator $L: \HH \ra \KK\x\EE$ by
\begin{equation*}
Lx=(Tx,Vx), \ \ \ \ x\in \HH.
\end{equation*}
The adjoint operator of $L$ with respect to the indefinite inner product $\K{\cdot}{\cdot}_\rho$ in $\KK\x\EE$ is given by
\[
L^\#(w,z)=T^\#w + \rho V^\#z, \quad (y,z)\in\KK\x\EE,
\]
and it is immediate that $L^\#L=T^\#T+ \rho V^\#V$.
Now consider the selfadjoint operator $G\in\mc{L}(\HH)$ given by
 \begin{equation}\label{eq:def_G}
 G:= (L^\#L)^{-1/2}V^\#V(L^\#L)^{-1/2}.
\end{equation}
Then $P(\lambda)$ can be rewritten as
\[
T^\#T+\lambda V^\#V=(L^\#L)^{1/2}\big(I+(\lambda-\rho)G\big)\,(L^\#L)^{1/2}.
\]
Hence, the operator pencil $T^\#T+\lambda V^\#V$ is congruent to the pencil $I+\gamma G$, where $\gamma=\lambda-\rho$.
If $\kappa=\frac{\rho_+-\rho_-}{2}$, then
$I+\gamma G$ is positive semidefinite if and only if $\gamma\in[-\kappa,\kappa]$, and positive definite if and only if
$\gamma\in(-\kappa,\kappa)$, see \cite[Prop. 3.11]{GZMMP21}.
This reduction technique is very common in the operator pencils context, since the auxiliary pencil $P'(\gamma)=I+\gamma G$ is easier to
analyze, see e.g. \cite{Gohberg}. A similar procedure is also applied in \cite{Hmam} for a constrained quadratic optimization problem
in a finite dimensional setting.

Consider the neutral elements of the quadratic form $x\mapsto\PI{Gx}{x}$, i.e.
\[
Q(G):=\set{x\in\HH\,:\,\PI{Gx}{x}=0}.
\]
Next, we determine sufficient conditions under which there exist $y\in Q(G) $ and
$\gamma\in[-\kappa,\kappa]$ such that
\begin{equation}\label{eq:nueva_eq_normal}
(I+\gamma G)y=u_0,
\end{equation}
for every vector $u_0\in\HH$.
Later on, we show that this implies that Problem \ref{pb 1} admits a solution for every initial data point.
Solving \eqref{eq:nueva_eq_normal} is equivalent
to finding the vectors in $ Q(G) $ which minimize the distance to the vector $u_0\in\HH$:
\[
\min\|y-u_0\|^2 \qquad \text{subject to} \qquad y\in  Q(G) .
\]
In fact, the normal equation \eqref{eq:nueva_eq_normal} is just the corresponding version of \eqref{eq:eqnormal_2} for this minimal
distance problem.
\medskip

Consider the canonical decomposition of $G$ as the difference of two positive operators: there exist unique subspaces
$\HH_\pm\subseteq\HH$ and positive definite operators $G_\pm\in\mc{L}(\HH_\pm)$ such that 
\begin{equation}\label{eq:descomposicion}
\HH=\HH_+\oplus\HH_-\oplus N(G),
\end{equation}
and $G=\left(\begin{smallmatrix}
G_+ & 0 & 0 \\
0 & -G_- & 0 \\
0 & 0 & 0
\end{smallmatrix}\right)$ with respect to \eqref{eq:descomposicion}.

If $u_0=u_0^++u_0^-+u_0^0$ with $u_0^\pm\in\HH_\pm$
and $u_0^0\in N(G)$, \eqref{eq:nueva_eq_normal} translates into
\begin{align}\label{ecu desag}
\left\{
\begin{array}{rcl}
(I_++\gamma G_+)y^+ &=& u_0^+ \\
(I_--\gamma G_-)y^- &=& u_0^- \\
y^0 &=& u_0^0
\end{array},
\right.
\end{align}
where $y=y^++y^-+y^0$ with $y^\pm\in\HH_\pm$ and $y^0\in N(G)$. 

\bigskip

Consider the subspaces
\begin{equation}\label{eq:def_nucleos}
\mc{N}_\pm:=N(I\mp\kappa G).
\end{equation}
It is easy to check that $\N_\pm=N(I_\pm-\kappa G_\pm)$.
Since $\N_\pm$ is invariant for $G_\pm\in\mc{L}(\HH_\pm)$, its orthogonal complement in $\HH_\pm$,
\[
\mc{D_\pm}:=\HH_\pm\ominus \N_\pm,
\]
is also an invariant subspace for $G_\pm$. We call $\mc{D_\pm}$ the \emph{positive (negative) defect subspace} of $\N_\pm$. 

\begin{lema}\label{lema:primera_serie}
Given $u\in\HH_\pm$, decompose it as $u=v+w$ with $v\in\mc{N}_\pm$ and $w\in\mc{D}_\pm$.
Then, for every $\tau\in(-\kappa,\kappa)$,
\begin{align*}
\|(I_\pm\pm \tau G_\pm)^{-1}u\|^2 
=\frac{\kappa^2}{(\kappa\pm \tau)^2}\|v\|^2+\|(I_\pm\pm \tau G_\pm)^{-1}w\|^2.
\end{align*}
\end{lema}


\begin{proof}
Given $\tau\in(-\kappa,\kappa)$, considering that $\|G_\pm\|=\tfrac{1}{\kappa}$ we have that
$I_+ + \tau G_+$ is invertible, see \cite[Prop. 3.11]{GZMMP21}. $\N_+$ and $\mc{D}_+$ are both invariant
subspaces for $I_+ + \tau G_+$. Also, since $\mc{N}_+=N(I_+-\kappa G_+)$, if $v\in \N_+$ then 
\[
(I_+ + \tau G_+)v= (I_+ -\kappa G_+)v + (\kappa + \tau)G_+ v = \frac{ \kappa + \tau}{\kappa} v.
\] 

Now, let $u=v+w\in\HH_+$ with $v\in\mc{N}_+$ and $w\in\mc{D}_+$. Then, 
\[
(I_+ + \tau G_+)^{-1}u= (I_+ + \tau G_+)^{-1}v + (I_+ + \tau G_+)^{-1}w= \frac{\kappa}{\kappa + \tau} v + (I_+ + \tau G_+)^{-1}w,
\]
and it is immediate that
\begin{align*}
\|(I_++\tau G_+)^{-1}u\|^2 =\frac{\kappa^2}{(\kappa+\tau)^2}\|v\|^2+\|(I_++\tau G_+)^{-1}w\|^2. 
\end{align*}
The proof of the remaining norm equality is similar.
\end{proof}

As a consequence of Lemma \ref{lema:primera_serie}, if $u=v+w\in\HH_\pm$ with $v\in\mc{N}_\pm$ and $w\in\mc{D}_\pm$ is such that
$v\neq0$, then $\lim_{\tau\to\mp\kappa}\|(I_\pm\pm\tau G_\pm)^{-1}u\|=+\infty$.

\begin{lema}\label{lema:en_el_rango} 
Given $u\in\HH_\pm$, if $\lim_{\tau\to\kappa}\|(I_\pm-\tau G_\pm)^{-1}u\|<+\infty$,
then $u\in R(I_\pm-\kappa G_\pm)$.
\end{lema}

\begin{proof} In the following we prove the statement for a vector in $\HH_+$, the proof for vectors in $\HH_-$ is analogous.
Let $u\in\HH_+$ be such that $\lim_{\tau\to \kappa}\|(I_+-\tau G_+)^{-1}u\|<+\infty$. Then
\[
u\in\mc{D}_+=\overline{R(I_+-\kappa G_+)}.
\]
Consider the sequence $(x_n)_{n\in \NN}$ in $\mc{D}_+$ defined by $x_n=\big(I_+-(\kappa-\frac{1}{n})G_+\big)^{-1}u$, with $n\in\NN$. By hypothesis,
$(x_n)_{n\in \NN}$ is bounded. Assume that $M>0$ is such that $\|x_n\|\leq M$ for every $n\in\NN$. Then,
\begin{equation}\label{signos}
\big\|(I_+-\kappa G_+)x_n-u\big\|=\frac{1}{n}\ \big\|G_+\big(I_+-(\kappa-\tfrac{1}{n})G_+\big)^{-1}u\big\|\leq 
\frac{M}{n} \|G_+\|\to 0.
\end{equation}

We claim that $(x_n)_{n\in \NN}$ is a Cauchy sequence. To prove it consider the sequence of positive definite operators $(\Delta_n)_{n\in \NN}$ defined by
\[
\Delta_n= \Big(I_+-\big(\kappa-\tfrac{1}{n}\big)G_+\Big)^{-1}.
\]
By the functional calculus for selfadjoint operators, given $m,n\in\N$, $m\geq n$ implies that $\Delta_n\leq\Delta_m$,
and $\Delta_n$ commutes with $\Delta_m$, see e.g. \cite{Conway}.
Then,
\begin{align*}
0\leq \|x_n\|^2 &=\PI{\Delta_nu}{\Delta_nu}=\PI{\Delta_n \big( \Delta_n^{1/2}u\big)}{\Delta_n^{1/2}u}  \\
&\leq \PI{\Delta_m \big( \Delta_n^{1/2}u\big)}{\Delta_n^{1/2}u} = \PI{\Delta_mu}{\Delta_nu}=\PI{x_m}{x_n}.
\end{align*}
Hence,
\begin{align*}
\|x_n-x_m\|^2 &= \|x_n\|^2 - 2\real\big(\PI{x_n}{x_m}\big) + \|x_m\|^2 \\ &\leq \|x_n\|^2-2\|x_n\|^2 + \|x_m\|^2=\|x_m\|^2-\|x_n\|^2 \to 0,
\end{align*}
as $m,n\ra\infty$. 

Since $\mc{D}_+$ is closed, there exists $x\in \mc{D_+}$ such that $\|x_n-x\|\ra 0$ as $n\ra\infty$. 
Thus, \eqref{signos} says that $u=(I_+ -\kappa G_+)x$, i.e. $u\in R(I_+ -\kappa G_+)$.
\end{proof}

\medskip

\begin{lema}\label{lema:necesaria_nucleos}
If there exist $y\in Q(G)$ and $\gamma\in[-\kappa,\kappa]$ such that $(I+\gamma G)y\in\HH_\pm\setminus\set{0}$, then $\N_\mp\neq\set{0}$.
\end{lema}

\begin{proof}
Suppose that $(I+\gamma G)y=u_0^-$, with $u_0\in\HH_-\setminus\set{0}$, $y\in Q(G)$ and $\gamma\in[-\kappa,\kappa]$, and assume that
$\N_+=\set{0}$. If $y=y^++y^-+y^0$ with $y^\pm\in\HH_\pm$ and $y^0\in N(G)$, by
\eqref{ecu desag} we get $y^+=0$ and $y^-\neq 0$. But since $y\in Q(G)$,
\[
0=\PI{Gy}{y}=\PI{Gy^-}{y^-}=-\PI{G_-y^-}{y^-}<0.
\]
Then $\N_+\neq\set{0}$.

 A similar argument for a vector $u_0^+\in\HH_+\setminus\set{0}$ proves that $\N_-\neq\set{0}$. 
\end{proof}

\medskip


\begin{prop}\label{prop:necesarias_suficientes}
For every $u_0\in\HH$ there exist $y_0\in Q(G) $ and $\gamma\in[-\kappa,\kappa]$ such that
\[
(I+\gamma G)y_0=u_0,
\]
if and only if $\mc{N}_+\neq\set{0}$ and $\mc{N}_-\neq\set{0}$.
\end{prop}

\begin{proof} The necessity follows by Lemma \ref{lema:necesaria_nucleos}.
To prove the converse, assume that $\mc{N}_+\neq\set{0}$ and $\mc{N}_-\neq\set{0}$. Let $u_0\in\HH$, and consider the
decomposition $u_0=u_0^++u_0^-+u_0^0$ with $u_0^\pm\in\HH_\pm$ and $u_0^0\in N(G)$, and the real valued functions $f_\pm$ defined by
\[
f_\pm(\tau)=\|G_\pm^{1/2}(I_\pm\pm\tau G_\pm)^{-1}u_0^\pm\|,\quad \tau\in(-\kappa,\kappa).
\]
If there exists $\tau_0\in(-\kappa,\kappa)$ such that $f_+(\tau_0)=f_-(\tau_0)$, then setting
$y_0=(I_++\tau_0G_+)^{-1}u_0^++(I_--\tau G_-)^{-1}u_0^-+u_0^0$ yields $y_0\in Q(G) $ and $(I+\tau_0 G)y_0=u_0$.

On the other hand, assume that $f_+(\tau)>f_-(\tau)$ for every $\tau\in(-\kappa,\kappa)$. By
the functional calculus for selfadjoint operators, 
$f_-$ is a monotone increasing function of $\tau$ on the interval $[0,\kappa)$.
Since the extension of $f_+$ is a continuous function of $\tau$ on the compact interval $[0,\kappa]$, it follows that
\[
\lim_{\tau\to\kappa}\|(I_--\tau G_-)^{-1}u_0^-\|\leq\lim_{\tau\to\kappa}\|G_-^{-1/2}\|\, f_-(\tau)<+\infty.
\]
Lemma \ref{lema:en_el_rango} then assures that $u_0^-\in R(I_--\kappa G_-)$. Now, since
$\mc{N}_-\neq\set{0}$, let us choose
$y\in\mc{N}_-$ with $\|y\|=1$. Hence, considering that $G_-y=\tfrac{1}{\kappa}y$
and $(I_--\kappa G_-)^\dag u_0^-\bot y$, and setting
\[
y_0=(I_++\kappa G_+)^{-1}u_0^++(I_--\kappa G_-)^\dag u_0^-+\alpha_-y+u_0^0,
\]
with
\begin{equation}\label{eq:alpha}
\alpha_-:=\bigg(\kappa\Big(\|G_+^{1/2}(I_++\kappa G_+)^{-1}u_0^+\|^2-\|G_-^{1/2}(I_--\kappa G_-)^\dag u_0^-\|^2\Big)\bigg)^{1/2},
\end{equation}
yields $y_0\in Q(G) $ and $(I+\kappa G)y_0=u_0$.

A similar argument holds if we assume that $f_+(\tau)<f_-(\tau)$ for every $\tau\in(-\kappa,\kappa)$, and thus the proof is
complete.
\end{proof}

\begin{lema}\label{lema:malditos_nucleos} Under Hypothesis \ref{hipo sec 4},
\begin{equation*}
\mc{N}_\pm\subseteq (L^\#L)^{1/2}\big(N(T)\big)^\bot.
\end{equation*}
\end{lema}

\begin{proof} We prove the statement for $\N_+$, a similar argument holds for $\N_-$.
On the one hand, from
\[
T^\#T+\rho_- V^\#V=(L^\#L)^{1/2}(I -\kappa G)(L^\#L)^{1/2},
\]
it follows that $\N_+=(L^\#L)^{1/2}\big(N(T^\#T+\rho_- V^\#V))$.

On the other hand, if $x\in N(T^\#T+\rho_-V^\#V)$ and $y\in N(T)$, then
\[
\rho_-\PI{x}{V^\#Vy}=\PI{x}{(T^\#T+\rho_-V^\#V)y}=\PI{(T^\#T+\rho_-V^\#V)x}{y}=0,
\]
i.e., $x\in V^\#V\big(N(T)\big)^\bot=L^\#L\big(N(T)\big)^\bot$. Hence, we have that
$N(T^\#T+\rho_-V^\#V)\subseteq L^\#L\big(N(T)\big)^\bot$. Applying $(L^\#L)^{1/2}$ to both sides of the inclusion,
\[
\N_+\subseteq (L^\#L)^{1/2}\Big((L^\#L)^{-1}\big(N(T)^\bot\big)\Big)=(L^\#L)^{1/2}\big(N(T)\big)^\bot. \qedhere
\]
\end{proof}

We are now in conditions to state the main result of this section, establishing the necessary and sufficient conditions for
Problem \ref{pb 1} to admit a solution for every initial data point. We no longer assume that
Hypothesis \ref{hipo sec 4} hold.

\begin{thm}\label{teo:ida_y_vuelta} Assume that $N(T)\cap N(V)=\set{0}$. The following conditions are equivalent:
\begin{enumerate}[label=\roman*)]
\item $\ZZ(w,z)\neq\varnothing$ for every $(w,z)\in\KK\x\EE$\textcolor{red};
\item there exists $\alpha>0$ such that $\K{Ty}{Ty}\geq\alpha\|y\|^2$
for every $y\in\CV$, and
\begin{equation}\label{eq:supinf}
\sup_{x\in\mc{P}^-(V)}\frac{\K{Tx}{Tx}}{\K{Vx}{Vx}}\quad\quad\text{and}\quad\quad\inf_{x\in\mc{P}^+(V)}\frac{\K{Tx}{Tx}}{\K{Vx}{Vx}}
\end{equation}
are attained.
\end{enumerate}
\end{thm}

\begin{proof}
{\it ii)$\to$i):} Suppose that item {\it ii} holds and let $(w_0,z_0)\in\KK\x\EE$. Since the supremum and infimum in \eqref{eq:supinf} being
attained is equivalent to $\N_+\neq\set{0}$ and $\N_-\neq\set{0}$, by Proposition \ref{prop:necesarias_suficientes} for every $u_0\in\HH$
there exist $\gamma\in[-\kappa,\kappa]$ and $\widetilde{y}_0\in Q(G)$ such that
\[
(I+\gamma G)\widetilde{y}_0=u_0.
\]
Setting $u_0=(L^\#L)^{-1/2}T^\#(TV^\dag z_0-w_0)$, applying $(L^\#L)^{1/2}$ to both sides of the equation, and taking
$y_0=(L^\#L)^{-1/2}\widetilde{y}_0$ and $\lambda=\gamma+\rho$ the result follows.

\smallskip

{\it i)$\to$ii):} Assume that $\ZZ(w,z)\neq\varnothing$ for every $(w,z)\in\KK\x\EE$.
By Proposition \ref{prop:vuelta}, it suffices to show that the infimum and supremum in \eqref{eq:supinf} are attained,
or equivalently, that $\mc{N}_+\neq\set{0}$ and $\mc{N}_-\neq\set{0}$.
By Theorem \ref{teo:derivada}, for every $w_0\in\KK$ there exist $\lambda\in[\rho_-,\rho_+]$ and $y\in\CV$ such that
$(T^\#T+\lambda V^\#V)y=T^\#w_0$. Equivalently, $(I+\gamma G)\widetilde{y}=u_0$, where $\gamma=\lambda-\rho\in[-\kappa,\kappa]$,
$\widetilde{y}=(L^\#L)^{1/2}y\in Q(G)$ and $u_0=(L^\#L)^{-1/2}T^\#w_0$. Hence, for every
\[
u_0\in(L^\#L)^{1/2}\big(N(T)\big)^\bot,
\]
there exist $\gamma\in[-\kappa,\kappa]$ and $\widetilde{y}\in Q(G)$
such that $(I+\gamma G)\widetilde{y}=u_0$.

We now show that $\HH_+\cap (L^\#L)^{1/2}(N(T))^\bot$ and 
$\HH_-\cap (L^\#L)^{1/2}(N(T))^\bot$ are non trivial subspaces, which by
Lemma \ref{lema:necesaria_nucleos} in turn implies that $\mc{N}_-\neq\set{0}$ and $\mc{N}_+\neq\set{0}$.
Let us assume that $\rho>0$. It holds that $(L^\#L)^{1/2}\big(N(T^\#T)\big)\subseteq\HH_+$. In fact,
$T^\#Tx=0$ if and only if $L^\#Lx=\rho V^\#V x$, or equivalently,
$(L^\#L)^{1/2}x=\rho (L^\#L)^{-1/2}V^\#Vx=\rho G(L^\#L)^{1/2}x$.
If $(L^\#L)^{1/2}x=x_+ + x_- + x_0$, with
$x_\pm\in \HH_\pm$ and $x_0\in N(G)$, then
\[
x_+=\rho G_+ x_+, \quad x_-=-\rho G_- x_-, \quad \text{and} \quad x_0=\rho\cdot 0. 
\] 
The last equation says that $x_0=0$, and $x_-=0$ because $\rho>0$ and $G_-\in \mc{L}(\HH_-)^+$.
Therefore, $(L^\#L)^{1/2}x=x_+$. Then
\[
(L^\#L)^{1/2}\big(N(T)\big)\subseteq(L^\#L)^{1/2}\big(N(T^\#T)\big)\subseteq\HH_+.
\]
Hence,
\[
\HH_-\subseteq(L^\#L)^{1/2}\big(N(T)\big)^\bot,
\]
and, by Lemma \ref{lema:necesaria_nucleos},
$\N_+\neq\set{0}$. But by Lemma \ref{lema:malditos_nucleos}
\[
\N_+\subseteq\HH_+\cap(L^\#L)^{1/2}\big(N(T)\big)^\bot.
\]
Then $\HH_+\cap(L^\#L)^{1/2}\big(N(T)\big)^\bot\neq0$, which implies that $\N_-\neq\set{0}$.
A similar argument holds for the case $\rho<0$, and thus the proof is complete.

\end{proof}

\section{Description of the set of solutions}\label{soluciones}

In this section we consider a selfadjoint operator $G\in \mc{L}(\HH)$. Decomposing it as the sum of two positive operators with orthogonal ranges
$G=G_+-G_-$, by \cite[Prop. 3.11]{GZMMP21} we have that $I + \gamma G$ is positive semidefinite if and only if
$\gamma\in[-\|G_+\|^{-1} ,\|G_-\|^{-1}]$, and it is positive definite if and only if $\gamma\in(-\|G_+\|^{-1},\|G_-\|^{-1})$.

For simplicity, we assume that $\kappa:=\|G_+\|=\|G_-\|$. Hence, $I+\gamma G$ is positive semidefinite if and only if  $\gamma\in [-\kappa,\kappa ]$. 

Also, we assume that the subspaces $\N_+$ and $\N_-$ given by \eqref{eq:def_nucleos} are non trivial, ensuring that for every $u\in\HH$ there exist $y\in Q(G)$ and $\gamma\in [-\kappa,\kappa ]$
such that
\begin{equation}\label{eq:eq_normal_sec_5}
(I+\gamma G)y=u.
\end{equation}


\medskip

From now we consider a fixed vector $u_0\in\HH$. If $(I+\gamma G)y=u_0$, for some $\gamma\in[-\kappa,\kappa]$ and $y\in Q(G)$,
then from \eqref{ecu desag} it holds that $y^0=u_0^0$, and
\begin{align*}
\left\{
\begin{array}{rcl}
(I_++\gamma G_+)y^+ &=& u_0^+ \\
(I_--\gamma G_-)y^- &=& u_0^-
\end{array}.
\right.
\end{align*}
If $u_0\in N(G)$ then $y=u_0$ is the unique solution. On the other hand,
if $u_0\notin N(G)$, then there is
a unique $\gamma\in[-\kappa,\kappa]$ for any solution, as the next proposition shows.

\begin{prop}\label{prop:unicidad} 
If $u_0\notin N(G)$ then there exists a unique $\gamma\in[-\kappa,\kappa]$ such that
$(I+\gamma G)y=u_0$ admits a solution $y\in Q(G)$.
\end{prop}

\begin{proof}
Let $u_0\notin N(G)$, and assume there exist $\gamma_1,\gamma_2\in[-\kappa,\kappa]$ and $y_1,y_2\in Q(G)$ such that
\begin{align}\label{eq:sistema}
(I+\gamma_1 G)y_1&=u_0,\nonumber\\
(I+\gamma_2 G)y_2&=u_0.
\end{align}
On the one hand, since $\PI{Gy_i}{y_i}=0$ for $i=1,2$, $\PI{u_0}{y_i}=\|y_i\|^2$. On the other hand,
\begin{align}\label{eq:sistema_2}
\|y_1\|^2=\PI{u_0}{y_1}=\PI{(I+\gamma_2 G)y_2}{y_1}=\PI{y_2}{y_1}+\gamma_2\PI{Gy_2}{y_1},\nonumber\\
\|y_2\|^2=\PI{u_0}{y_2}=\PI{(I+\gamma_1 G)y_1}{y_2}=\PI{y_1}{y_2}+\gamma_1\PI{Gy_1}{y_2}.
\end{align}
This implies that
\begin{equation}\label{eq:gammas}
(\gamma_1-\gamma_2)\PI{Gy_1}{y_2}=\|y_1\|^2-\|y_2\|^2.
\end{equation}
By Cauchy-Schwarz inequality,
\begin{align*}
\|y_1\|^2&=|\PI{u_0}{y_1}|=|\PI{(I+\gamma_2 G)y_2}{y_1}|\leq\|y_2\|\|y_1\|,\\
\|y_2\|^2&=|\PI{u_0}{y_2}|=|\PI{(I+\gamma_1 G)y_1}{y_2}|\leq\|y_1\|\|y_2\|,
\end{align*}
and consequently $\|y_1\|=\|y_2\|$. By \eqref{eq:gammas}, this implies that $\gamma_1=\gamma_2$ or $\PI{Gy_1}{y_2}=0$.
However, if $\PI{Gy_1}{y_2}=0$, then from \eqref{eq:sistema_2} it is easy to see that $y_1=y_2$, which in turn, by \eqref{eq:sistema}, implies that
$(\gamma_1-\gamma_2)Gy_1=0$. But $y_1\notin N(G)$ because $u_0\notin N(G)$, and hence $\gamma_1=\gamma_2$.
\end{proof}

\medskip

For $u_0\notin N(G)$, consider the set of solutions to
\eqref{eq:eq_normal_sec_5},
\[
\Theta:=\set{y\in Q(G) \,:\,(I+\gamma G)y=u_0},
\]
for the unique suitable $\gamma\in[-\kappa,\kappa]$.
The following proposition describes the structure of the set $\Theta$, depending on whether
$\gamma$ is an interior point of the interval or $\gamma= \pm\kappa$. Denote by $\St$ the unit sphere
in $\HH$, i.e.
\[
\St=\setb{x\in\HH\,\,:\,\,\|x\|=1}.
\]
\begin{lema}\label{lema:forma_soluciones}
Let $u_0\notin N(G)$ and consider the unique $\gamma\in[-\kappa,\kappa]$ given by Proposition \ref{prop:unicidad}.
\begin{enumerate}[label=\roman*)]
\item If $\gamma\in(-\kappa,\kappa)$, then
\[
\Theta=\set{(I+\gamma G)^{-1} u_0}.
\]
\item If $\gamma=\kappa$, then there exists $\alpha_-\geq 0$ such that
\[
\Theta=(I+\kappa G)^\dag u_0\,+\alpha_-\cdot\mc{N}_-\cap\St.
\]
\item If $\gamma=-\kappa$, then there exists $\alpha_+\geq0$ such that
\[
\Theta=(I-\kappa G)^\dag u_0\,+\alpha_+\cdot\mc{N}_+\cap\St.
\]
\end{enumerate}
\end{lema}

\begin{proof}
{\it i)} If $\gamma\in(-\kappa,\kappa)$ then $I+\gamma G$ is invertible. Hence, $y_0=(I+\gamma G)^{-1} u_0$.

\smallskip

{\it ii)} Suppose that $\gamma=\kappa$. Since
\[
 Q(G)=\setB{y=y^++y^-+y^0\,\,:\,\,\|G_+^{1/2}y^+\|=\|G_-^{1/2}y^-\|\,,\,y^\pm\in\HH_\pm,y^0\in N(G)},
\]
writing $u_0=u_0^++u_0^-+u_0^0$ the condition $(I+\kappa G)y_0=u_0$ leads to
\[
(I_++\kappa G_+)y_0^+=u_0^+\quad\text{,}\quad(I_--\kappa G_-)y_0^-=u_0^-\quad\text{and}\quad y_0^0=u_0^0.
\]
Then,
\[
y_0^+=(I_++\kappa G_+)^{-1}u_0^+\quad\text{and}\quad y_0^-=(I_--\kappa G_-)^\dag u_0^-+ v,
\]
where $v\in \N_-$. If $v=0$, set $\alpha_-=0$. Otherwise, if $v\neq 0$, setting $\alpha_-:=\|v\|>0$ and
$y_-:=\frac{v}{\|v\|}\in\mc{N}_-\cap\St$, we have that 
\[
y_0 =(I+\kappa G)^\dag u_0 +\alpha_- y_-.
\]
It only remains to show that $\alpha_-$ is the same for every $y\in \Theta$.
But, since $\|G_+^{1/2}y^+\|=\|G_-^{1/2}y^-\|$, $\alpha_-$ is given by \eqref{eq:alpha},
and it does not depend on $y_0$ but only on $u_0$. Thus, 
\[
\Theta=(I+\kappa G)^\dag u_0\,+\alpha_-\cdot\mc{N}_-\cap\St.
\]
An analogous procedure for the case $\gamma=-\kappa$ completes the proof.
\end{proof}

As a consequence, as in Example \ref{ejemplo indef}, we can describe $\Theta$ by only analyzing which components of $u_0$ are null
according to the decomposition $\HH=\mc{N}_+\oplus\mc{D}_+\oplus \mc{N}_-\oplus\mc{D}_-\oplus N(G)$.

\begin{prop}\label{prop:lambda_conocido} 
Consider $u_0\notin N(G)$ and write $u_0=v^++w^++v^-+w^-+u_0^0$,
with $v^\pm\in\mc{N}_\pm$, $w^\pm\in\mc{D}_\pm$ and $u_0^0\in N(G)$.
\begin{enumerate}[label=\roman*)]
\item If $u_0\in\HH_\pm$, then there exists $\alpha_\pm>0$ such that
\[
\Theta=(I\pm\kappa G)^\dag u_0\,+\alpha_\mp\cdot\mc{N}_\mp\cap\St.
\]
\item If $v^+\neq0$ and $v^-\neq0$, then $\gamma\in(-\kappa,\kappa)$ and
\[
\Theta=\set{(I+\gamma G)^{-1}u_0}.
\]
\end{enumerate}
\end{prop}

\begin{proof}
\noi {\it i)}\ Assume that $u_0\in\HH_+$ and consider $y_0=y_0^++y_0^-+y_0^0\in\Theta$ with $y_0^\pm\in\HH_\pm$ and
$y_0^0\in N(G)$. If $\gamma\in[-\kappa,\kappa]$ is such that $(I+\gamma G)y_0=u_0$, then
\[
(I_++\gamma G_+)y_0^++(I_--\gamma G_-)y_0^-+y_0^0=u_0=v^++w^++u_0^0.
\]
Since $(I_--\gamma G_-)y_0^-=0$ and $y_0^-\neq 0$, it holds that $\gamma=\kappa$. The result then follows
from Lemma \ref{lema:forma_soluciones}. The proof is analogous when $u_0\in\HH_-$.
\medskip

\noi {\it ii)}\ Assuming that $v^+\neq0$ and $v^-\neq0$, following the same ideas of Proposition \ref{prop:necesarias_suficientes},
we show that there exists $\gamma\in(-\kappa,\kappa)$ such that
\[
\|G_+^{1/2}(I_++\gamma G_+)^{-1}(v^++w^+)\|=\|G_-^{1/2}(I_--\gamma G_-)^{-1}(v^-+w^-)\|,
\]
which implies that the vector $y_0:= (I+\gamma G)^{-1}u_0$ belongs to $\Theta$ (because $y_0\in Q(G) $ and $(I+\gamma G)y_0=u_0$).


Consider the real valued functions $g_\pm$ defined by
\[
g_\pm(\tau)=\|G_\pm^{1/2}(I_\pm\pm \tau G_\pm)^{-1}(v^\pm+w^\pm)\|^2,\quad\quad \tau\in(-\kappa,\kappa).
\]
Since $G_\pm^{1/2}$ and $(I_\pm\pm\tau G_\pm)^{-1}$ commute, and $G_\pm^{1/2}v^\pm=\kappa^{-1/2}v^\pm$,
Lemma \ref{lema:primera_serie} implies that
\begin{align*}
g_\pm(\tau)&=\frac{\kappa^2}{(\kappa\pm \tau)^2}\|G_\pm^{1/2}v^\pm\|^2+\|(I_\pm\pm \tau G_\pm)^{-1}G_\pm^{1/2}w^\pm\|^2\\
&=\frac{\kappa}{(\kappa\pm \tau)^2}\|v^\pm\|^2+\|G_\pm^{1/2}(I_\pm\pm \tau G_\pm)^{-1}w^\pm\|^2,\quad\quad\text{for every $\tau\in(-\kappa,\kappa)$}.
\end{align*}
Since the operator $I_-+\kappa G_-$ is invertible, it follows that $g_-$ is bounded on $(-\kappa,0)$. Analogously, $g_+$ is bounded
on $(0,\kappa)$. On the other hand, since $v^\pm\neq0$, it is immediate that
\[
\lim_{\tau\to-\kappa}g_+(\tau)=+\infty\quad\text{and}\quad\lim_{\tau\to\kappa}g_-(\tau)=+\infty.
\]
Hence, it is readily seen that there exists $\gamma\in(-\kappa,\kappa)$ such that $g_-(\gamma)=g_+(\gamma)$, or equivalently,
\[
\|G_+^{1/2}(I_++\gamma G_+)^{-1}(v^++w^+)\|=\|G_-^{1/2}(I_--\gamma G_-)^{-1}(v^-+w^-)\|.
\]
Thus, $\Theta=\set{(I+\gamma G)^{-1} u_0}$.
\end{proof}

As it is illustrated by Case 3 in Example \ref{ejemplo indef}, if $u_0$ does not belong to $\HH_+$ nor to $\HH_-$ and also
$v^-=0$ or $v^+=0$ (which is the only situation not covered by Proposition \ref{prop:lambda_conocido}), it is not possible to
assert whether $\Theta$ is a singleton.

\bigskip

To end this section, we show how these previous results can be applied to describe the set of solutions to Problem 1.
We assume that $N(T)\cap N(V)=\set{0}$ and $\ZZ(w,z)\neq\varnothing$ for every $(w,z)\in\KK\x\EE$.

\medskip

Consider an initial data point $(w_0,z_0)\in\KK\x\EE$ and a fixed vector $x_0\in\HH$ such that $Vx_0=z_0$. By Theorem
\ref{teo:derivada}, the set of solutions to Problem \ref{pb 1} is $\ZZ(w_0,z_0)=x_0+\Omega$ with the set $\Omega$ given by  
\[
\Omega:=\setB{y\in\CV\,\,:\,\,(T^\#T+\lambda V^\#V)y=-T^\#(Tx_0-w_0)\,\,\text{ for some $\lambda\in[\rho_-,\rho_+]$}}.
\]
Considering the operator
$G$ given by \eqref{eq:def_G} and setting $u_0:=-(L^\#L)^{-1/2}T^\#(Tx_0-w_0)$, $\Omega$ can be alternatively described
as
\[
\Omega=\setB{y\in\CV\,\,:\,\,(I+\gamma G)(L^\#L)^{1/2}y=u_0\,\,\text{ for some $\gamma\in[-\kappa,\kappa]$}},
\]
Since $(L^\#L)^{1/2}(\CV)=Q(G)$, it follows that $\Omega=(L^\#L)^{-1/2}(\Theta)$ and thus
\[
\ZZ(w_0,z_0)=x_0+(L^\#L)^{-1/2}(\Theta).
\]

\medskip

We establish now the main result of this section.

\begin{thm}
There exists an open and dense subset $\mc{M}$ of $\KK\x\EE$ such that $\ZZ(w,z)$ is a singleton for every $(w,z)\in\mc{M}$.
\end{thm}

\begin{proof} 
The set
\[
\mc{\widetilde{M}}=\setB{u=v^+ +w^+ + v^- + w^-+u^0\in \HH\,:\,\,v^\pm\in\mc{N}_\pm\setminus\set{0}\,, w^\pm\in\mc{D}_\pm\,, u^0\in N(G)\,}
\]
is non empty, open and dense in $\HH$. In fact, $\mc{\widetilde{M}}$ is non empty as a
consequence of the assumption that $\mc{N}_\pm\neq\varnothing$
and Lemma \ref{lema:malditos_nucleos},
while the remaining conditions follow immediately. By Proposition \ref{prop:lambda_conocido}, $\Theta$ is a singleton for every
$u\in\mc{\widetilde{M}}$.

Finally, considering the operator $A:\KK\x\EE\to\HH$ given by
\[
A(w,z)=-(L^\#L)^{-1/2}T^\#(TV^\dag z-w),\quad\quad (w,z)\in\KK\x\EE,
\]
yields $R(A)=(L^\#L)^{1/2}\big(N(T)\big)^\bot$ is a closed subspace, and consequently $\mc{M}:=A^{-1}\big(\widetilde{\mc{M}}\big)$
is an open and dense subset of $\KK\x\EE$. Hence, $\ZZ(w,z)$ is a singleton for every $(w,z)\in \mc{M}$.
\end{proof}

\medskip

\begin{obs} 
An immediate consequence of Hypothesis \ref{hipo sec 4} is that $N(T)\cap N(V)=\set{0}$.
However, the condition in this hypothesis can be slightly modified in order to address the case in which this intersection is non trivial.
Indeed, the following conditions are equivalent:
\begin{enumerate}[label={\it \roman*})]
\item $\ZZ(w,z)\neq\varnothing$ for every $(w,z)\in\KK\x\EE$;
\item there exists $\alpha>0$ such that
\[
\K{Ty}{Ty}\geq\alpha\|y\|^2,\quad\quad\text{for every $y\in \CV\cap\big(N(T)\cap N(V)\big)^\bot$,}
\]
and
\[
\sup_{x\in\mc{P}^-(V)}\frac{\K{Tx}{Tx}}{\K{Vx}{Vx}}\quad\quad\text{and}\quad\quad\inf_{x\in\mc{P}^+(V)}\frac{\K{Tx}{Tx}}{\K{Vx}{Vx}}
\]
are attained.
\end{enumerate}
As a result, there exists an open and dense subset $\mc{M}$ of $\KK\x\EE$ such that,
instead of a singleton, the set of solutions to Problem \ref{pb 1} is an affine manifold parallel to the subspace
$N(T)\cap N(V)$ i.e. for every $(w,z)\in\mc{M}$,
\[
\ZZ(w,z)=\wt{x}_{(w,z)} + N(T)\cap N(V),
\]
where $\wt{x}_{(w,z)}$ is a particular solution to Problem \ref{pb 1} with initial data $(w,z)$. 
%
\end{obs}

\section{Application: Indefinite abstract mixed splines}\label{mixed}


The abstract mixed problem in Hilbert spaces was originally proposed by A.~I. Rozhenko and V.~A. Vasilenko in \cite{[RV]}, and it can be stated as follows.
Let $(\HH,\PI{\cdot}{\cdot}_\HH)$, $(\KK_1,\PI{\cdot}{\cdot}_{\KK_1})$, $(\KK_2,\PI{\cdot}{\cdot}_{\KK_2})$ and $(\EE,\PI{\cdot}{\cdot}_\EE)$ be  Hilbert spaces, and consider (bounded) surjective operators  $U:\HH\ra\KK_1$, $W:\HH\ra\KK_2$ and $V:\HH\ra\EE$.
Given $(w_0,z_0)\in\KK_2\x\EE$ and $\mu\in\RR$, analize the existence of
\begin{equation*}
\min_{x\in\HH} \parentesis{\|Ux\|_{\KK_1}^2+\mu\|Wx-w_0\|_{\KK_2}^2},\quad\textit{ \normalfont{subject to} }Vx=z_0,
\end{equation*}
and if the minimum exists, find the set of arguments at which it is attained.

The abstract mixed splines problem is a generalization of the abstract interpolating and smoothing splines problems proposed by Atteia in \cite{[Att1]}.
For a complete exposition on these subjects see \cite{[Att4], Be, [CLM1]}.

Generalizations to Krein spaces of the abstract interpolating and smoothing splines problems have been studied before \cite{GMMP10, JOTA}.
in particular, a generalization of the abstract mixed splines problem was also proposed in \cite{GMMP10}.

The following indefinite abstract mixed splines problem follows as a natural generalization of this family of problems. 
Given a Hilbert space $(\HH,\PI{\cdot}{\cdot}_\HH)$, and Krein spaces $(\KK_1,\K{\cdot}{\cdot}_{\KK_1})$, $(\KK_2,\K{\cdot}{\cdot}_{\KK_2})$ and $(\EE,\K{\cdot}{\cdot}_\EE)$,
let $U\in \mc{L}(\HH,\KK_1)$, $W\in \mc{L}(\HH,\KK_2)$ and $V\in\mc{L}(\HH,\EE)$ be (bounded) surjective operators.  

\begin{prob}\label{pb mixed 1}
Given $\mu\neq0$, and $(w_0,z_0)\in\KK_2\x\EE$, analyze the existence of
\begin{align*}
\min_{x\in\HH} &\left(\K{Ux}{Ux}_{\KK_1}+ \mu\K{Wx-w_0}{Wx-w_0}_{\KK_2}\right), \\ &\textit{ \normalfont{subject to} }\K{Vx-z_0}{Vx-z_0}_{\EE}=0,
\end{align*}
and if the minimum exists, find the set of arguments at which it is attained.
\end{prob}

If  $V^\#V$ is semidefinite then Problem \ref{pb mixed 1}
becomes the abstract mixed splines problem analyzed in \cite{GMMP10}.
We proceed now to describe how this problem can be studied in the context of the ILSP analyzed in this paper, in the case when $V^\#V$ is indefinite.

Given $\mu\neq 0$, define the inner product on $\KK_1\x \KK_2$ as in \eqref{met indef para KxE} and assume that $U^\#U+\mu W^\#W$ is indefinite.
Also, defining the operator $T: \HH \ra \KK_1\x\KK_2$ by
\begin{equation}\label{T mixed}
Tx:=(Ux,Wx), \ \ \ \ x\in \HH,
\end{equation}
it is immediate that Problem \ref{pb mixed 1} is equivalent to the following: given $(w_0,z_0)\in\KK_2\x\EE$, analyze the existence of
\begin{equation}\label{cuadr min mixed}
\min_{x\in\HH} \K{Tx-(0,w_0)}{Tx-(0,w_0)}_\mu,\textit{ \normalfont{subject to} }\K{Vx-z_0}{Vx-z_0}_{\EE}=0,
\end{equation}
and if the minimum exists, find the set of arguments at which it is attained. Hence, it is clear that this is a particular case of Problem \ref{pb 1}. 
Moreover, if $w_0=0$ and $T$ is surjective, then \eqref{cuadr min mixed} reduces to the indefinite abstract splines problem 
considered in \cite{JOTA} with initial data $z_0\in \EE$.
The following proposition provides a necessary and sufficient condition for this particular case.

\begin{prop}
The operator $T$ defined in \eqref{T mixed} is surjective if and only if 
\[
\HH=N(U)+N(W).
\]
\end{prop}

\begin{proof} Assume $R(T)=\KK_1\x\KK_2$, and let $(u,0)\in\KK_1\x\KK_2$. Then there exists $y\in\HH$ such that $(Uy,0)=Ty=(u,0)$. Consequently, $y\in N(W)$ and since $u\in\KK_1$
is arbitrary $\KK_1=U(N(W))$ follows. Thus, $\HH=U^{-1}\parentesis{U(N(W))}=N(U)+N(W)$.

Conversely, assume that $\HH=N(U)+N(W)$. Then $N(U)^\bot\cap N(W)^\bot=\set{0}$. Given $(u,w)\in N(T^\#)$, we have that
\[
U^\#u=-\mu W^\#w\in R(U^\#)\cap R(W^\#)=N(U)^\bot\cap N(W)^\bot=\set{0}, 
\]
and $(u,w)=(0,0)$ because $U^\#$ and $W^\#$ are injective.  
Therefore, $N(T^\#)=\set{0}$ and $\overline{R(T)}=\KK_1\x\KK_2$.

Since $U(N(W))=U(\HH)=\KK_1$, it follows that $U(N(W))$ is closed. Now,
consider a sequence $(x_n)_{n\in\NN}$ in $\HH$ such that $Tx_n\ra (y,z)$ for some $(y,z)\in \KK_1\x\KK_2$. Then, for each
$n\in\NN$ consider $u_n=P_{N(W)^\bot}x_n$ and $v_n=P_{N(W)}x_n$. Then, $Wu_n=Wx_n=z_n\ra z$ and, since $u_n\in N(W)^\bot$, $u_n=W^\dag Wu_n\ra W^\dag z$. Therefore,
$Uu_n\ra UW^\dag z$ and
\[
Uv_n=Ux_n - Uu_n\ra y-UW^\dag z.
\]
The closedness of $U(N(W))$ implies that there exists $u\in N(W)$ such that $y-UW^\dag z=Uu$. Hence, $U(W^\dag z + u)=y$ and $W(W^\dag z +u)=z+Wu=z$. Thus,
$T(W^\dag z + u)=(y,z)$ and the range of $T$ is closed, thus completing the proof.
\end{proof}

Now for a fixed $\rho\neq 0$ we define a new indefinite inner product on $\KK_1\x\KK_2\x \EE$. If $u,u'\in\KK_1$, $w,w'\in \KK_2$ and $z,z'\in\EE$, 
\begin{align*}
\K{(u,w,z)}{(u',w',z')}_{\rho} 
:=\K{u}{u'}_{\KK_1}+\mu\K{w}{w'}_{\KK_2}+\rho\K{z}{z'}_\EE.
\end{align*}

It is easy to see that the space $\KK_1\x\KK_2\x\EE$ is a Krein space with this indefinite inner product.
Also, defining the operator $L: \HH \ra \KK_1\x\KK_2\x\EE$ by
\begin{equation}\label{L mixed}
Lx:=(Tx,Vx)=(Ux,Wx,Vx), \ \ \ \ x\in \HH,
\end{equation}
it is immediate that
\[
L^\#L=U^\#U+\mu W^\#W+\rho V^\#V.
\]

By means of the operators $T$ and $L$ defined in \eqref{T mixed} and \eqref{L mixed} respectively, the results concerning the ILSP analyzed in this paper can be directly applied.


\begin{thebibliography}{XXXXXX}


\bibitem{Mardsen} R. Abraham, J. Mardsen, T. Ratiu, {\it Manifolds, Tensor Analysis and Applications}, Addison Wesley, London, 1983.

\bibitem{Ando} T. Ando, {\it Linear operators on Krein spaces}, Hokkaido University, Sapporo, Japan, 1979.



\bibitem{[Att1]} M. Atteia,  {\it G\'eneralization de la d\'efinition et des propiet\'es des "splines fonctions"}, C.R. Sc. Paris 260 (1965), 3550--3553.

\bibitem{[Att4]} M. Atteia, {\it Hilbertian kernels and spline functions}, North-Holland Publishing Co., Amsterdam, 1992. 

\bibitem{Azizov} T.~Ya.~Azizov and I.~S.~Iokhvidov, {\it Linear Operators in spaces with an indefinite metric}, John Wiley and sons, 1989.

\bibitem{BG} A. Ben-Israel,T. N. E. Greville, Generalized inverses. Theory and applications. Second edition. Springer-Verlag, New York, 2003.  

\bibitem{Be} A.Yu. Bezhaev, V.A. Vasilenko, {\it Variational Theory of Splines}, Kluwer Academic/Plenum Publishers, New York, 2001.

\bibitem{Bognar} J. Bogn\'ar, {\it Indefinite Inner Product Spaces, Springer-Verlag}, 1974.



\bibitem{Boltyanski} V. Boltyanski, H. Martini and V. Soltan, {\it Geometric methods and optimization problems}, Kluwer, Dordrecht (1999).

\bibitem{Boyd} S. Boyd and V. Lieven, {\it Convex Optimization}, Cambridge University Press (2004).


\bibitem{Canu2004} S. Canu, C. S. Ong, X. Mary, A. Smola, {\it Learning with non-positive kernels},
Proc. of the 21st International Conference on Machine Learning (2004), 639--646.

\bibitem{Canusplines} S. Canu, C. S. Ong, X. Mary, {\it Splines with non positive kernels},
Proceedings of the 5th International ISAAC Congress (2005), 1--10.


\bibitem{[CLM1]} R. Champion, C. T. Lenard,  T. M. Mills, {\it An introduction to abstract splines},
Math. Scientist 21 (1996), 8--26.


\bibitem{Conway} J. B. Conway, {\it A Course in Functional Analysis}, Springer, 1990.


\bibitem{Dong} V. V. Dong,  N. N. Nguyen, {\it On the Solution Existence of Nonconvex Quadratic Programming Problems in Hilbert Spaces},
Acta Math Vietnam 43, 155-174 (2018).

\bibitem{Douglas} R. G. Douglas, {\it On majorization, factorization and range inclusion of operators in Hilbert space}, Proc.
Amer. Math. Soc. 17 (1996), 413--416.

\bibitem{Dritschel} M. A. Dritschel and J. Rovnyak, {\it Operators on indefinite inner product spaces}, Fields Institute Monographs no. 3, Amer. Math. Soc.
Edited by Peter Lancaster (1996), 141--232.



\bibitem{Gartner} T. Gärtner, J. W. Lloyd, P. A. Flach, {\it Kernels for Structured Data},
International Conference on Inductive Logic Programming ILP: Inductive Logic Programming, 66-83 (2003).


\bibitem{Gohberg} I. Gohberg, S. Goldberg, and M. Kaashoek, {\it Classes of Linear Operators Vol. I}, Birkh\"{a}user Verlag (1990).

\bibitem{GMMP10} J. I. Giribet, A. Maestripieri, and F. Mart\'{\i}nez Per\'{\i}a, {\it Abstract splines in Krein spaces}, J. Math. Anal. Appl. 369 (2010), 423--436.

\bibitem{GMMP10_2} J. I. Giribet, A. Maestripieri, and F. Mart\'{\i}nez Per\'{\i}a, {\it A Geometrical Approach to Indefinite Least Squares Problems}, Acta Appl. Math. 111:1 (2010), 65--81.

\bibitem{GMMP16} J. I. Giribet, A. Maestripieri, and F. Mart\'{\i}nez Per\'{\i}a, {\it Indefinite least-squares problems and pseudo-regularity}, J. Math. Anal. Appl. 430 (2016), 895--908.

\bibitem{JOTA} S. Gonzalez Zerbo, A. Maestripieri, F. Mart\'inez Per\'ia, {\it Indefinite Abstract Splines with a Quadratic Constraint},
186, 209--225 (2020).

\bibitem{GZMMP21} S. Gonzalez Zerbo, A. Maestripieri, and F. Mart\'{\i}nez Per\'{\i}a, {\it Linear pencils and quadratic programming problems with a quadratic constraint}, submitted.


\bibitem{Haasdonk2005} B. Haasdonk, D. Keysers, {\it Tangent distance kernels for support vector machines}, Proc. of the 16th Int. Conf. on
Pattern Recognition 2, 864-868 (2002).





\bibitem{Hmam} H. Hmam, {\it Quadratic optimisation with one quadratic equality constraint},
Technical Report DSTO-TR-2416, Electronic Warfare \& Radar Division, Defence Science \& Technology Organisation (2010).


\bibitem{Kelly} R. J. Kelly, W. A. Thompson, {\it Quadratic Programming in Real Hilbert Spaces}, Journal of the Society for Industrial
and Applied Mathematics (1964).

\bibitem{IKL82} I.~S.~Iohvidov, M.~G.~Krein, and H.~Langer, {\it Introduction to the Spectral Theory of Operators in Spaces with an Indefinite Metric}, Akademie-Verlag, Berlin, 1982.






\bibitem{Luenberger} D. G. Luenberger, {\it Optimization by Vector Space Methods}, John Wiley, New York, 1969.


\bibitem{Mohri} M. Mohri, A. Rostamizadeh, A. Talwalkar, Foundations of Machine Learning, MIT Press, (2012). 

\bibitem{Nashed} M. Z. Nashed, {\it Inner, outer, and generalized inverses in Banach and Hilbert spaces}, Numer. Funct. Anal. Optim. 9 (1987), 261--325.

\bibitem{Oglic_2} D. Oglic, T. Gaertner, {\it Learning in Reproducing Kernel Krein Spaces}, Proceedings of the 35th International Conference on
Machine Learning, PMLR 80  (2018), 3859--3867.
\bibitem{Oglic} D. Oglic, T. Gaertner, {\it Scalable Learning in Reproducing Kernel Krein Spaces}, Proceedings of the 36th International Conference
on Machine Learning, vol. 97, PMLR  (2019), 4912--4921.



\bibitem{Pencils} H. Palanthandalam-Madapusi, T. Van Pelt and D. Bernstein, {\it Matrix pencils and existence conditions for quadratic programming with a
sign-indefinite quadratic equality constraint}, Computational Optimization and Applications 45 (2009), 533--549.

\bibitem{Park} J. Park and S. Boyd, {\it General Heuristics for Nonconvex Quadratically Constrained Quadratic Programming}, arXiv:1703.07870 (2017).


\bibitem{Polik} I. P\'olik, T. Terlaky, {\it A Survey of the S-Lemma}, SIAM Review 49 (2007), 371--418.

\bibitem{Powell} M.J.D. Powell and Y. Yuan, {\it A trust-region algorithm for equality constrained optimization}, Mathematical
Programming 49 (1991), 189--211.

\bibitem{Rockafellar} R. T. Rockafellar, {\it Convex Analysis}, Princeton University Press (1970).

\bibitem{Rovnyak} J. Rovnyak, {\it Methods on Krein space operator theory, Interpolation theory, systems theory and related
topics} (Tel Aviv/Rehovot, 1999), Oper. Theory Adv. Appl. 134 (2002), 31--66.


\bibitem{[RV]} A. I. Rozhenko, V. A.  Vasilenko,  {\it Variational approach in abstract splines: achievements and open problems},
East J. Approx. 1 (1995), 277--308.




\bibitem{Signoretto} M. Signoretto, K. Pelckmans, J. Suykens, {\it Quadratically Constrained Quadratic Programming for Subspace
Selection in Kernel Regression Estimation}, Conference: Artificial Neural Networks - ICANN 2008, 18th International Conference,
Prague, Czech Republic, Proceedings, Part I (2008), 175--184.

\bibitem{Sonnenberg2006} S. Sonnenberg, G. Ratsch, C. Schafer, B. Scholkopf, {\it Large scale multiple kernel learning},
Journal of Machine Learning Research 7 (2006), 1531--1565.

\bibitem{Sundaram} R. K. Sundaram, {\it A First Course in Optimization Theory}, Cambridge University Press (1996).


\bibitem{Xia} Y. Xia, S. Wang and R. L Sheu, {\it S-lemma with equality and its applications}, Mathematical Programming 156 (2016), 513--547.

\bibitem{Xu} J. Xu, A. Paiva, I. M. Park, Il , J. C. Principe, {\it A Reproducing Kernel Hilbert Space Framework for Information-Theoretic
Learning}, IEEE Transactions on Signal Processing 56 (2009), 5891--5902.

\bibitem{Ye} Y. Ye and S. Zhang, {\it New results on quadratic minimization}, SIAM Journal on Optimization 14 (2003), 245--267.



\end{thebibliography}
\end{document}